\documentclass[11pt]{amsart}

\usepackage{amsmath}
\usepackage{amssymb}
\usepackage{amscd}
\usepackage{mathrsfs}
\usepackage{color}

\usepackage{xy}
\xyoption{all}

\usepackage{ytableau}
\usepackage{hyperref}

\newcommand{\nc}{\newcommand}

\nc{\kk}{{\mathsf{k}}}

\nc{\SB}{{\mathsf{B}}}
\nc{\SJ}{{\mathsf{J}}}
\nc{\SK}{{\mathsf{K}}}
\nc{\SL}{{\mathsf{L}}}
\nc{\SM}{{\mathsf{M}}}
\nc{\SO}{{\mathsf{O}}}
\nc{\SQ}{{\mathsf{Q}}}
\nc{\SR}{{\mathsf{R}}}
\nc{\ST}{{\mathsf{T}}}
\nc{\SU}{{\mathsf{U}}}

\nc{\C}{{\mathbb{C}}}
\nc{\LL}{{\mathbb{L}}}
\nc{\PP}{{\mathbb{P}}}
\nc{\QQ}{{\mathbb{Q}}}
\nc{\RR}{{\mathbb{R}}}
\nc{\ZZ}{{\mathbb{Z}}}

\nc{\ba}{{\mathbf{a}}}
\nc{\bm}{{\mathbf{m}}}
\nc{\bq}{{\mathbf{q}}}
\nc{\bu}{{\mathbf{u}}}
\nc{\BB}{{\mathbf{B}}}
\nc{\BC}{{\mathbf{C}}}
\nc{\BD}{{\mathbf{D}}}
\nc{\BG}{{\mathbf{G}}}
\nc{\BH}{{\mathbf{H}}}
\nc{\BK}{{\mathbf{K}}}
\nc{\BL}{{\mathbf{L}}}
\nc{\BM}{{\mathbf{M}}}
\nc{\BP}{{\mathbf{P}}}
\nc{\BS}{{\mathbf{S}}}
\nc{\BT}{{\mathbf{T}}}
\nc{\BU}{{\mathbf{U}}}
\nc{\BZ}{{\mathbf{Z}}}
\nc{\BPr}{{\mathsf{P}}}
\nc{\BR}{{\mathbf{R}}}
\nc{\BW}{{\mathbf{W}}}

\nc{\CA}{{\mathscr{A}}}
\nc{\CB}{{\mathcal{B}}}
\nc{\CC}{{\mathcal{C}}}
\nc{\D}{{\mathcal{D}}}
\nc{\CE}{{\mathcal{E}}}
\nc{\CF}{{\mathcal{F}}}
\nc{\CG}{{\mathcal{G}}}
\nc{\CH}{{\mathcal{H}}}
\nc{\CI}{{\mathcal{I}}}
\nc{\CJ}{{\mathscr{J}}}
\nc{\CK}{{\mathscr{K}}}
\nc{\CL}{{\mathcal{L}}}
\nc{\CM}{{\mathcal{M}}}
\nc{\CN}{{\mathcal{N}}}
\nc{\CO}{{\mathcal{O}}}
\nc{\CP}{{\mathcal{P}}}
\nc{\CQ}{{\mathcal{Q}}}
\nc{\CR}{{\mathcal{R}}}
\nc{\CS}{{\mathscr{S}}}
\nc{\CT}{{\mathcal{T}}}
\nc{\CU}{{\mathcal{U}}}
\nc{\CV}{{\mathcal{V}}}
\nc{\CW}{{\mathcal{W}}}
\nc{\CX}{{\mathcal{X}}}
\nc{\CY}{{\mathcal{Y}}}

\nc{\fa}{{\mathfrak{a}}}
\nc{\fb}{{\mathfrak{b}}}
\nc{\fd}{{\mathfrak{d}}}
\nc{\fg}{{\mathfrak{g}}}
\nc{\fn}{{\mathfrak{n}}}
\nc{\fp}{{\mathfrak{p}}}
\nc{\fu}{{\mathfrak{u}}}

\nc{\FA}{{\mathfrak{A}}}
\nc{\FB}{{\mathfrak{B}}}
\nc{\FD}{{\mathfrak{D}}}
\nc{\FE}{{\mathfrak{E}}}
\nc{\FL}{{\mathfrak{L}}}
\nc{\FM}{{\mathfrak{M}}}

\nc{\lotimes}{\mathbin{\mathop{\otimes}\limits^{\mathbb{L}}}}

\nc{\CExt}{\mathop{\mathcal{E}\mathit{xt}}\nolimits}
\nc{\CHom}{\mathop{\mathcal{H}\mathit{om}}\nolimits}
\nc{\CEnd}{\mathop{\mathcal{E}\mathit{nd}}\nolimits}
\nc{\RCHom}{\mathop{\mathsf{R}\mathcal{H}\mathit{om}}\nolimits}

\nc{\Hom}{\mathop{\mathsf{Hom}}\nolimits}
\nc{\Ext}{\mathop{\mathsf{Ext}}\nolimits}
\nc{\End}{\mathop{\mathsf{End}}\nolimits}
\nc{\RGamma}{\mathop{{\mathsf{R}}\Gamma}\nolimits}
\nc{\RHom}{\mathop{\mathsf{RHom}}\nolimits}

\nc{\Tor}{\mathop{\mathsf{Tor}}\nolimits}

\nc{\Hilb}{\mathop{\mathsf{Hilb}}\nolimits}
\nc{\Spec}{\mathop{\mathsf{Spec}}\nolimits}
\nc{\Proj}{\mathop{\mathsf{Proj}}\nolimits}
\nc{\Pic}{\mathop{\mathsf{Pic}}\nolimits}

\nc{\Mod}{\mathop{\mathsf{Mod}}\nolimits}
\nc{\grmod}{\mathop{\mathsf{grmod}}\nolimits}
\nc{\Grmod}{\mathop{\mathsf{Grmod}}\nolimits}
\nc{\Qcoh}{\mathop{\mathsf{Qcoh}}\nolimits}

\nc{\Ann}{\mathop{\mathsf{Ann}}\nolimits}
\nc{\Ker}{\mathop{\mathsf{Ker}}\nolimits}
\nc{\Coker}{\mathop{\mathsf{Coker}}\nolimits}

\nc{\Cone}{\mathop{\mathsf{Cone}}\nolimits}
\nc{\Tot}{{\mathsf{Tot}}}

\nc{\coh}{{\mathop{{\mathsf{coh}}}}}
\nc{\Ab}{{\mathop{\mathcal{A}\mathit{b}}}}

\nc{\Tr}{\mathop{\mathsf{Tr}}\nolimits}
\nc{\Ind}{\mathop{\mathsf{Ind}}\nolimits}
\nc{\Res}{\mathop{\mathsf{Res}}\nolimits}

\nc{\Conv}{\mathop{\mathsf{Conv}}\nolimits}

\nc{\codim}{\mathop{\mathsf{codim}}\nolimits}

\nc{\Sing}{{\mathsf{Sing}}}
\nc{\supp}{\mathop{\mathsf{supp}}}
\nc{\vol}{\mathop{\mathsf{vol}}\nolimits}
\nc{\ch}{\mathop{\mathsf{ch}}\nolimits}
\nc{\perf}{{\mathsf{perf}}}
\nc{\rk}{\mathop{\mathsf{rk}}}
\nc{\Pf}{{\mathsf{Pf}}}
\nc{\Gr}{{\mathsf{Gr}}}
\nc{\OGr}{{\mathsf{OGr}}}
\nc{\OFl}{{\mathsf{OFl}}}
\nc{\TGr}{{\widetilde{\mathsf{Gr}}}}
\nc{\TOGr}{{\widetilde{\mathsf{OGr}}}}
\nc{\ti}{{\tilde{\imath}}}
\nc{\tj}{{\tilde{\jmath}}}
\nc{\Flag}{{\mathsf{Fl}}}
\nc{\Kosz}{{\mathsf{Kosz}}}
\nc{\LGr}{{\mathsf{LGr}}}
\nc{\LFl}{{\mathsf{LFl}}}
\nc{\SGr}{{\mathsf{SGr}}}
\nc{\OF}{{\mathsf{OF}}}
\nc{\Fl}{{\mathsf{Fl}}}
\nc{\Bl}{{\mathsf{Bl}}}

\nc{\Gm}{{\mathbb{G}_m}}
\nc{\GL}{{\mathsf{GL}}}
\nc{\PGL}{{\mathsf{PGL}}}
\nc{\GSL}{{\mathsf{SL}}}
\nc{\SP}{{\mathsf{Sp}}}
\nc{\SSO}{{\mathsf{SO}}}
\nc{\Spin}{{\mathsf{Spin}}}

\nc{\fsl}{{\mathfrak{sl}}}
\nc{\fso}{{\mathfrak{so}}}
\nc{\fgl}{{\mathfrak{gl}}}

\nc{\ev}{{\mathsf{ev}}}
\nc{\coev}{{\mathsf{coev}}}
\nc{\tr}{{\mathsf{tr}}}
\nc{\id}{{\mathsf{id}}}
\nc{\opp}{{\mathsf{opp}}}

\nc{\ed}[1]{{\mathsf{e}}({#1})}
\nc{\expd}[3]{{\mathsf{e}}_{(#1,#2)}({#3})}
\nc{\gexpd}[4]{{\mathsf{e}}_{(#1,#2),#3}({#4})}
\nc{\ced}[1]{{\mathsf{c}}({#1})}

\nc{\topd}[1]{{\mathsf{h}}({#1})}
\nc{\btmd}[1]{{\mathsf{t}}({#1})}

\nc{\You}[2]{\mathrm{YD}_{#1, #2}}
\nc{\Dcone}[1]{\mathrm{P}_{#1}}
\nc{\Yous}[1]{\mathrm{SYD}^{sym}_{#1}}

\nc{\Youe}[2]{\mathrm{EYD}_{#1, #2}}
\nc{\Youel}[3]{\Youe{#1}{#2}^{#3}}

\nc{\Yourb}[1]{\mathrm{BYD}^{r}_{#1}}

\nc{\hlf}{\frac{1}{2}}
\nc{\veps}{\varepsilon}

\nc{\FFF}{{\mathrm{FFL}}}
\newcommand{\dlen}{{\ell_{\mathrm{diag}}}}

\newcommand{\Clif}{\operatorname{\mathscr{C}\!\ell}}

\DeclareMathOperator{\Sym}{\mathrm{Sym}}

\theoremstyle{plain}

\newtheorem{theorem}{Theorem}[section]
\newtheorem{conjecture}[theorem]{Conjecture}
\newtheorem{lemma}[theorem]{Lemma}
\newtheorem{proposition}[theorem]{Proposition}
\newtheorem{corollary}[theorem]{Corollary}

\theoremstyle{definition}

\newtheorem{example}[theorem]{Example}

\theoremstyle{remark}

\newtheorem{remark}[theorem]{Remark}

\title{Derived categories of curves as components of Fano manifolds}
\author{Anton Fonarev}
\author{Alexander Kuznetsov}
\address{\sloppy
\parbox{0.95\textwidth}{
{\bf A.F.: }Algebraic Geometry Section, Steklov Mathematical Institute of Russian Academy of Sciences,
8 Gubkin str., Moscow 119991 Russia
\hfill\\[2pt]
\phantom{{\bf A.F.: }}National Research University Higher School of Economics, Russian Federation,
Laboratory of Mirror Symmetry, NRU HSE, 6 Usacheva str., Moscow, Russia, 119048
\hfill
}\bigskip}
\email{avfonarev@mi.ras.ru}
\address{\sloppy
\parbox{0.95\textwidth}{
{\bf A.K.: }Algebraic Geometry Section, Steklov Mathematical Institute of Russian Academy of Sciences,
8 Gubkin str., Moscow 119991 Russia
\hfill\\[2pt]
\phantom{{\bf A.K.: }}The Poncelet Laboratory, Independent University of Moscow
\hfill\\[2pt]
\phantom{{\bf A.K.: }}Laboratory of Algebraic Geometry, National Research University Higher School of Economics
\hfill
}\bigskip}
\email{akuznet@mi.ras.ru}
\thanks{Both authors were partially supported by the Russian Academic Excellence Project ``5-100'' and
by the RFBR grants 15-01-02164 and 15-51-50045.
A.K.\ was partially supported by the Simons foundation. 
A.F.\ was partially supported by the RFBR grant 15-01-02158 and by the Young Russian Mathematics award.}

\begin{document}

\begin{abstract}
We prove that the derived category $\BD(C)$ of a generic curve of genus greater than one
embeds into the derived category $\BD(M)$ of the moduli space $M$ of rank~two stable
bundles on $C$ with fixed determinant of odd degree.
\end{abstract}

\maketitle

\section{Introduction}

For a smooth projective algebraic variety $X$ we denote by $\BD(X)$ the bounded derived category of coherent sheaves on $X$.
Let $C$ be a smooth projective curve of genus $g > 1$ over an algebraically closed field $\kk$ of characteristic 0. 
Fix a line bundle $L$ of odd degree on $C$ and let $M = M_C(2,L)$ be the moduli space of rank 2 stable vector bundles on $C$ with determinant $L$. 
Recall that $M$ is a smooth Fano variety of dimension $3g-3$ and index~2 (see, e.g., \cite[Th\'eor\`eme~F]{Drezet}).
Our goal is to show that the derived category of $C$ embeds fully and faithfully into the derived category of $M$.

Recall that by a theorem proved by Orlov in~\cite[Theorem~2.2]{Orlov} any embedding between the derived
categories of smooth projective varieties must be given by a Fourier--Mukai functor, i.e., by an object on their product.
A natural candidate for such an object in our case is, of course, the universal bundle $\CE$ on~$C\times M$, also known as the Poincar\'e bundle.

The main result of the paper is the following theorem, where the word \emph{generic} means the existence of a nonempty open substack in the moduli stack of genus $g$ curves, for which the statement holds.

\begin{theorem}\label{theorem-main}
For generic curve $C$ of genus $g > 1$ and a line bundle $L$ of odd degree the Fourier--Mukai functor
given by the universal vector bundle $\CE$ on $C \times M_C(2,L)$
\begin{equation*}
\Phi_\CE:\BD(C) \to \BD(M_C(2,L))
\end{equation*}
is fully faithful. 
In particular, $\BD(C)$ is an admissible subcategory in the derived category $\BD(M_C(2,L))$ of the Fano manifold~$M_C(2,L)$.
\end{theorem}

Our strategy for the proof of Theorem~\ref{theorem-main} is the following.
First we show that the statement holds for hyperelliptic curves using
an explicit description of the moduli 
space $M$ due to Desale--Ramanan and a spinor interpretation of the universal bundle.
Then we use a deformation argument to deduce the result for generic curve $C$.

The proof in the hyperelliptic case is a generalization of a theorem of Bondal and Orlov stating that the derived category of a hyperelliptic curve of genus $g > 1$ 
embeds into the derived category of a smooth complete intersection $X$ of two quadrics of dimension $2g$ (see~\cite{BondalOrlov}).
We show that, more generally, such an embedding exists for the Hilbert scheme $X_k$ of linear $(k-1)$-dimensional subspaces lying on~$X$ for any $1 \leq k \le g-1$.
The scheme $X_k$ can be identified with the intersection of two orthogonal isotropic Grassmannians $\OGr(k,2g+2)$ in $\Gr(k,2g+2)$,
and the Fourier--Mukai kernel that we use to construct the functor is formed by the family of the restrictions to $X_k$ of the spinor bundles on the Grassmannians~$\OGr(k,2g+2)$.
When~$k = 1$, this functor coincides with the functor used by Bodal and Orlov, and when $k = g - 1$, we show that this functor coincides with the functor 
given by the universal family of stable bundles of rank~2 on~$C$, which proves Theorem~\ref{theorem-main} in the hyperelliptic case.

Our Theorem~\ref{theorem-main} gives a partial answer to the following general conjecture which we learned from Sergey Galkin.

\begin{conjecture}\label{conj:galkin}
The functor $\Phi_\CE:\BD(C) \to \BD(M)$ is fully faithful for any smooth curve of genus $g > 1$.
\end{conjecture}

In fact, this conjecture would give a positive answer in dimension one to a general question
asked by Alexei Bondal some time ago:
whether for any algebraic variety $X$ there is Fano variety $Y$ and a fully faithful embedding $\BD(X) \to \BD(Y)$
(see~\cite{bernardara2016,kiem2017all,kiem2015fano} for other results in this direction).

The paper is organized as follows.
In Section~\ref{section-general} we explain the proof of Theorem~\ref{theorem-main} and reduce it to some
cohomological computations on isotropic and usual Grassmannians. 
In Section~\ref{section-bbw} we perform the required cohomological computations.
In particular, we find nice resolutions for the pushforwards of the spinor bundles
from isotropic Grassmannians to usual Grassmannians (Proposition~\ref{lemma-resolution}) which are interesting per se.

When the first draft of the paper was written we were informed that a similar result was
established by Narasimhan in a completely different way. 
% As far as we know, 
He proves Conjecture~\ref{conj:galkin} for all curves of genus $g \geq 4$, 
but, surprisingly, not for curves of smaller genus, see~\cite{narasimhan}.

The first author is grateful to the Max Planck Institute for Mathematics in Bonn, where his part of the present work was mostly completed in Spring~2015.

\section{The Main Theorem}
\label{section-general}

As already mentioned, the base field $\kk$ is algebraically closed of zero characteristic. 
All the functors between derived categories that we consider (pullbacks, pushforwards, tensor products) are derived.

\subsection{Intersection of isotropic Grassmannians}
\label{ss:ogr}

Let $C$ be a hyperelliptic curve of genus $g > 1$ and let $f:C \to \PP^1$ be
the hyperelliptic covering.
Choose a coordinate $t$ on $\PP^1$ such that the point $t = \infty$ is not on the branch divisor, and let $a_1,a_2,\dots,a_{2g+2}$ be the coordinates of the branch points.

Consider a vector space $V$ of dimension $2g + 2$ with coordinates $x_1,\dots,x_{2g+2}$ 
and the pencil of quadrics in $\PP(V)$ generated by
\begin{align*}
q_0(x_1,x_2,\dots,x_{2g+2}) &= a_1x_1^2 + a_2x_2^2 + \dots + a_{2g+2}x_{2g+2}^2,\\
q_\infty(x_1,x_2,\dots,x_{2g+2}) &= - (x_1^2 + x_2^2 + \dots + x_{2g+2}^2).
\end{align*}
For each $P \in \PP^1$ we denote by $Q_P$ the corresponding quadric in the pencil,
i.e.\ if $P=(s:t)$, then $Q_P$ is defined by the equation $q_P = sq_0+tq_\infty$.
Remark that there are exactly $2g + 2$ degenerate quadrics, and all the degenerate quadrics
are of corank~1.
We denote by $P_i = (1:a_i)$, $1 \le i \le 2g+2$, the points on $\PP^1$ corresponding to the degenerate quadrics in the pencil.
These points of $\PP^1$ will be referred to as {\em branching points}, and the other points as {\em non-branching}.

% Consider the intersection of quadrics in the pencil
% \begin{equation*}
% X =\{ Q_0 = Q_\infty = 0 \} \subset \PP(V).
% \end{equation*}
% Recall that by~\cite[Theorem~2.7]{BondalOrlov} there is a fully faithful embedding of derived categories $\BD(C) \to \BD(X)$. Our goal is to generalize this result.

For each $1 \le k \le g$ and each point $P \in \PP^1$ consider the orthogonal isotropic Grassmannian 
\begin{equation*}
\OGr_P(k,V) := \{ U \in \Gr(k,V)\ |\ \PP(U) \subset Q_P \},
\end{equation*}
which parametrizes linear $(k-1)$-dimensional spaces on the quadric $Q_P$.

If $P$ is a non-branching point, the quadric $Q_P$ is non-singular and $\OGr_P(k,V)$ is a homogeneous variety of the simple algebraic group $\Spin(V)$.
If $P$ is branching, 
the quadratic form $q_P$ on $V$ is degenerate of corank 1, and the variety $\OGr_P(k,V)$ is no longer homogeneous.
Below we discuss its structure.

Fix a vector $v_P \in V$ generating the kernel of the quadratic form $q_P$ and denote by 
\begin{equation*}
\pi_P \colon V \to V_P := V/v_P
\end{equation*}
the natural projection.
Let $q'_P$ be the non-degenerate quadratic form induced on $V_P$ by $q_P$, 
and let $Q'_P \subset \PP(V_P)$ be the corresponding quadric.
Let $\OGr(k,V_P)$ be the corresponding orthogonal isotropic Grassmannian, and
consider the subvariety of $\OGr_P(k,V) \times \OGr(k,V_P)$ defined by
\begin{equation*}
\TOGr_P(k,V) = \{ (U_k,U'_k) \in \OGr_P(k,V) \times \OGr(k,V_P)\ |\ U_k \subset \pi_P^{-1}(U'_k) \}.
\end{equation*}
It comes with natural projections $p_1:\TOGr_P(k,V) \to \OGr_P(k,V)$ and $p_2:\TOGr_P(k,V) \to \OGr(k,V_P)$.
Denote by $\CU_k$ and $\CU'_k$ the tautological bundles on $\OGr_P(k,V)$ and $\OGr(k,V_P)$ respectively, as well as their pullbacks to $\TOGr_P(k,V)$. 
Set $H = c_1(\CU^\vee_k)$ and $H' = c_1(\CU^{\prime\vee}_k)$. 
When $k\leq g-1$ the line bundles $\CO(H)$ and~$\CO(H')$ are the ample generators of the Picard groups of $\OGr_P(k,V)$ and $\OGr(k,V_P)$ respectively.

\begin{proposition}
\label{proposition-togr}
The map $p_2$ is a $\PP^k$-fibration, and in particular $\TOGr_P(k,V)$ is a smooth projective variety.
The map $p_1$ is birational with the exceptional divisor $E$ isomorphic to the space of partial isotropic flags
\begin{equation*}
E \cong \OFl(k-1,k;V_P) \subset \TOGr_P(k,V),
\end{equation*}
and its map to $\OGr_P(k,V)$ can be represented as the composition
\begin{equation*}
E \cong \OFl(k-1,k;V_P) \to \OGr(k-1,V_P) \hookrightarrow \OGr_P(k,V),
\end{equation*}
where the first map is the natural projection and the last embedding identifies $\OGr(k-1,V_P)$ with the subscheme of $\OGr_P(k,V)$ parameterizing $U_k \subset V$ isotropic for $Q_P$ and containing $v_P \in V$.

Moreover, in the Picard group of $\TOGr_P(k,V)$ we have a linear equivalence
\begin{equation}\label{equation-ehh}
E = H - H'.
\end{equation}
\end{proposition}

We gather the varieties and maps from the proposition above on the following commutative diagram
\begin{equation}\label{diagram-togr}
\vcenter{\xymatrix{
& E = \OFl(k-1,k;V) \ar[dl]_{} \ar[r]^-{} & \TOGr_P(k,V) \ar[dl]_{p_1} \ar[dr]^{p_2} \\
\OGr(k-1,V_P) \ar[r] & \OGr_P(k,V) && \OGr(k,V_P)
}}
\end{equation}
In fact, one can show that the subscheme $\OGr(k-1,V_P)$ is the singular locus of $\OGr_P(k,V)$ and the birational map $p_1$ is its blowup.
However, we will not need this, so we leave this statement as an exercise.

\begin{proof}
If $U'_k \subset V_P$ is a $q'_P$-isotropic subspace of dimension $k$, then its preimage $\pi_P^{-1}(U'_k)$ is a $(k+1)$-dimensional $q_P$-isotropic subspace in $V$. 
Thus, any $k$-dimensional subspace in $\pi_P^{-1}(U'_k)$ gives a point of~$\TOGr_P(k,V)$. 
Therefore the fibers of $p_2$ are $\Gr(k,\pi_P^{-1}(U'_k)) \cong \Gr(k,k+1) \cong \PP^k$.
To be more precise, 
\begin{equation}\label{equation-togr-relative}
\TOGr_P(k,V) \cong \Gr_{\OGr(k,V_P)} (k,\CO \oplus \CU'_k) \cong \PP_{\OGr(k,V_P)} ((\CO \oplus \CU'_k)^\vee).
\end{equation}
This proves the first claim.

Now consider the morphism $p_1:\TOGr_P(k,V) \to \OGr_P(k,V)$, $(U_k,U'_k) \mapsto U_k$.
On the open subset of~$\OGr_P(k,V)$ given by the condition $v_P \not\in U_k$
the morphism $p_1$ has a section given by $U_k \mapsto (U_k,\pi_P(U_k))$. This section
and $p_1$ are mutually inverse isomorphisms over this set, hence $p_1$ is birational.
Note also that if $v_P \in U_k$, then $\pi_P(U_k)$ has dimension $k-1$, hence the fiber of $p_1$ over such $U_k$ is a quadric in the space $(\pi_P(U_k))^\perp / \pi_P(U_k)$ 
(the orthogonal is taken with respect to the nondegenerate form $q'_P$) of dimension $(2g+1) - 2(k-1) = 2g-2k+3$.
It follows that the indeterminacy locus of $p_1^{-1}$ is isomorphic to~$\OGr(k-1,V_P)$, and the exceptional divisor is isomorphic to $\OFl(k-1,k;V_P)$. 
% Since $\TOGr_P(k,V)$ is smooth and the fibers of the exceptional 
% divisor are quadrics, it follows that the indeterminacy locus is the singular locus and $p_1$
% is a resolution of singularities.

To express the class $E$ of the exceptional divisor
in terms of $H$ and $H'$, note that under the identification~\eqref{equation-togr-relative}
the exceptional divisor is nothing but
\begin{equation}\label{equation-e-relative}
E = \Gr_{\OGr(k,V_P)} (k-1,\CU'_k) \subset \Gr_{\OGr(k,V_P)} (k,\CO \oplus \CU'_k) = \TOGr_P(k,V)
\end{equation}
Therefore, it coincides with the degeneracy locus of the natural morphism $\CU_k \xrightarrow{\ \pi_P\ } \CU'_k$ on $\TOGr_P(k,V)$.
Moreover, from~\eqref{equation-togr-relative} it is clear that the corank of this morphism on $E$ equals 1,
hence
\begin{equation*}
E = c_1(\CU'_k) - c_1(\CU_k) = -H' + H
\end{equation*}
which is precisely the relation we want.
\end{proof}

For each $1 \le k \le g$ we define the scheme
\begin{equation*}
X_k := \{ U \in \Gr(k,V)\ |\ \PP(U) \subset X \} = \bigcap_{P\in\PP^1} \OGr_P(k,V),
\end{equation*}
parameterizing all $k$-dimensional subspaces in $V$ isotropic for all quadrics in the pencil generated by $q_0$ and $q_\infty$.
Clearly, the scheme $X_k$ can be also written as the intersection $X_k = \OGr_{P_1}(k,V) \cap \OGr_{P_2}(k,V)$ for any pair of distinct points $P_1,P_2 \in \PP^1$.

\begin{lemma}\label{lemma-avoid}
For any branching point $P \in \PP^1$ the scheme $X_k$ avoids the indeterminacy locus of the rational map $p_1^{-1} \colon \OGr_P(k,V) \dashrightarrow \TOGr_P(k,V)$.
% the singularities of $\OGr_P(k,V)$.
\end{lemma}
\begin{proof}
As it was shown in the proof of Proposition~\ref{proposition-togr}, the indeterminacy locus
% singular locus of $\OGr_P(k,V)$ 
consists precisely of the isotropic subspaces containing $v_P$. 
However, the vector $v_P$ is not isotropic for the quadratic form $q_\infty$, hence is not contained in any subspace parameterized by $X_k$.
% does not lie on the quadric $Q_\infty$ from the pencil, hence not on any linear subspace on 
\end{proof}

It follows from the previous lemma that for any branching point $P$ one has a commutative diagram
\begin{equation}\label{diagram-maps-ti-tj}
\vcenter{\xymatrix@C=5em{
& \TOGr_P(k,V) \ar[dr]^{\tj_P} \ar[d]^{p_1} \\
X_k \ar[r]_{i_P} \ar[ur]^{\ti_P} & \OGr_P(k,V) \ar[r]_{j_P}  & \Gr(k,V)
}}
\end{equation}
where $i_P$ and $j_P$ denote the natural embeddings, $\tj_P = j_P \circ p_1$, and $\ti_P = p_1^{-1} \circ i_P$.
To be able to write things in a uniform way, for any nonbranching point $P \in \PP^1$ put
\begin{equation*}
\TOGr_P(k,V) := \OGr_P(k,V),
\qquad
\ti_P := i_P,
\qquad 
\tj_P := j_P.
\end{equation*}
Then the diagram~\eqref{diagram-maps-ti-tj} makes sense for all points $P \in \PP^1$.

The scheme $X_k$ was investigated by Miles Reid in his thesis, in particular he proved the following

\begin{lemma}[\protect{\cite[Theorem~2.6]{MilesReid}}]\label{lemma:x-smooth}
For every $1 \le k \le g$ the scheme $X_k$ is a smooth variety of dimension 
\begin{equation*}
\dim(X_k) = k(2g - 2k + 1).
\end{equation*}
\end{lemma}

We will also need the following

\begin{lemma}\label{lemma-tor-independence}
For any points $P_1 \ne P_2$ we have
\begin{equation*}
\TOGr_{P_1}(k,V) \times_{\Gr(k,V)} \TOGr_{P_2}(k,V) = X_k,
\end{equation*}
and the fiber product is Tor-independent.
Moreover, for any $P\in\PP^1$ the subvariety $\ti_P(X_k) \subset \TOGr_P(k,V)$ is the zero locus 
of a regular section of the vector bundle~$\Sym^2\CU_k^\vee$.
\end{lemma}
\begin{proof}
The maps $\ti_{P_1}$ and $\ti_{P_2}$ induce a natural map $X_k\to \TOGr_{P_1}(k, V)\times_{\Gr(k,V)}\TOGr_{P_2}(k, V)$. On the other hand, the projections $\TOGr_{P_i}(k,V)\to\OGr_{P_i}(k,V)$ induce a morphism
\begin{equation*}
   \TOGr_{P_1}(k, V)\times_{\Gr(k,V)}\TOGr_{P_2}(k, V)\to\OGr_{P_1}(k, V)\times_{\Gr(k,V)}\OGr_{P_2}(k, V)
   = X_k.
\end{equation*}
It is immediate from Proposition~\ref{proposition-togr}
% Lemma~\ref{lemma-avoid} 
that these morphisms are mutually inverse.
Tor-independence follows from the fact that~$X_k$ has expected dimension and from~\cite[Corollary 2.27]{KuznetsovHyperplane}.
Finally, the last statement follows from Proposition~\ref{proposition-togr}
% Lemma~\ref{lemma-avoid} 
and the fact that each isotropic Grassmannian is the zero locus of a regular section of the vector bundle $\Sym^2\CU_k^\vee$ on $\Gr(k,V)$.
\end{proof}

Below we show that $X_k$ is a Fano variety when $k \le g - 1$. 
This has a useful implication about the Picard group of the product $X_k \times C$.

\begin{lemma}\label{lemma:pic-x-c}
If $k \le g-1$ then $X_k$ is a Fano variety and 
\begin{equation*}
\Pic(X_k \times C) \cong \Pic(X_k) \oplus \Pic(C).
\end{equation*}
In particular, every line bundle on~$X_k \times C$ is an exterior tensor product of a line bundle on $X_k$ and a line bundle on~$C$.
\end{lemma}
\begin{proof}
By definition $X_k$ is the zero locus of a regular (by Lemma~\ref{lemma:x-smooth}) global section of the vector bundle~$\Sym^2\CU_k^\vee \oplus \Sym^2\CU_k^\vee$ on the Grassmannian $\Gr(k,V) \cong \Gr(k,2g+2)$.
The canonical class of the Grassmannian is $\CO(-2g-2)$, and the first Chern class of $\Sym^2\CU_k^\vee$ is $k+1$. 
Therefore, the canonical class of $X_k$ is
\begin{equation*}
- 2g - 2 + 2k + 2 = -2(g-k),
\end{equation*}
and for $k \le g-1$ it is anti-ample.
It follows from Kodaira vanishing that $H^{1,0}(X_k) = H^{0,1}(X_k) = 0$, hence $\Pic(X_k \times C) \cong \Pic(X_k) \oplus \Pic(C)$.
\end{proof}

Note that for $k = g$ the statement of the previous lemma is no longer true.
Indeed, by~\cite[Theorem~4.8]{MilesReid} and~\cite[Theorem~2]{DR} we have $X_g \cong \Pic^0(C)$ and the Poincar\'e line bundle on $\Pic^0(C) \times C$ is not an exterior tensor product.

\subsection{Spinor bundles}\label{ss:spinor}

Recall that if $V$ is a vector space of even dimension $2n$ and $q$ is a nondegenerate quadratic form on $V$, for each $k \le n - 1$ 
the isotropic orthogonal Grassmannian $\OGr(k,V)$ is a homogeneous space of the simple algebraic group $\Spin(V)$ of type $D_n$.
If $k \le n - 2$ then $\OGr(k,V)$ corresponds to the maximal parabolic subgroup $\BP_k \subset \Spin(V)$ associated with the vertex $k$ of the Dynkin diagram~$D_n$.
If $k = n - 1$, it corresonds to the submaximal parabolic $\BP_{n-1,n} = \BP_{n-1} \cap \BP_n$.

For every $k \le n-2$ the Grassmannian $\OGr(k,V)$ carries a pair of irreducible $\Spin(V)$-equivariant vector bundles $(S_{k,+},S_{k,-})$ such that 
\begin{equation*}
r(S_{k,\pm}) = 2^{n-1-k},
\qquad 
\det(S_{k,\pm}) = \CO(2^{n-2-k}H),
\end{equation*}
where as before $H = c_1(\CU_k^\vee)$ is the ample generator of $\Pic(\OGr(k,V))$, and $\CU_k$ is the tautological bundle on $\OGr(k,V)$.
% $\CO(H)$ is the restriction of the hyperplane class under the Pl\"ucker embedding of $\OGr(k,V)$ to $\Gr(k,V)$
% (for $k \le n-2$ it is a generator of the Picard group of $\OGr(k,V)$).

Analogously, if $V'$ is a vector space of odd dimension $2n + 1$ and $q'$ is a nondegenerate quadratic form on $V'$, for each $k \le n$ 
the isotropic orthogonal Grassmannian $\OGr(k,V')$ is a homogeneous space of the simple algebraic group $\Spin(V')$ of type $B_n$
corresponding to the maximal parabolic $\BP_k \subset \Spin(V')$.
This time for every $k \le n - 1$ the Grassmannian $\OGr(k,V')$ carries one irreducible $\Spin(V')$-equivariant spinor vector bundle $S_k$ such that 
\begin{equation*}
r(S_k) = 2^{n-k},
\qquad 
\det(S_k) = \CO(2^{n-1-k}H'),
\end{equation*}
where as before $H' = c_1(\CU_k^{\prime\vee})$ is the ample generator of $\Pic(\OGr(k,V'))$, and $\CU'_k$ is the tautological bundle on~$\OGr(k,V')$.
For $k = n$ the spinor bundle generates $\Pic(\OGr(k,V'))$ and $\CO(H') \cong S_k^{\otimes 2}$.

The bundles $S_{k,\pm}$ (in the even-dimensional case) as well as the bundle $S_k$ (in the odd-dimensional case) are called {\sf spinor bundles}.
They were defined in~\cite{ottaviani1988spinor} as pushforwards of line bundles from isotropic flag varieties.
We are going to use instead a description of spinor bundles in terms of the Clifford algebra 
of the natural family of quadrics over the orthogonal Grassmannian given in~\cite[Section~6]{KuznetsovIsotropic}.

% \medskip 

Now let us return to the pencil of quadrics $Q_P$ discussed in the previous subsection.
For every point $P \in \PP^1$ we have defined a variety $\TOGr_P(k,V)$ that is equal to $\OGr_P(k,V)$ if $P$ is non-branching and is a $\PP^k$-fibration over $\OGr(k,V_P)$ if $P$ is branching.
We define the following vector bundles on these varieties:
\begin{itemize}
\item 
if $P$ is non-branching, $S_{P,k,\pm}$ is the corresponding spinor bundle on $\TOGr_P(k,V) = \OGr_P(k,V)$;
\item 
if $P$ is branching, $S_{P,k} = p_2^*S_k$ is the pullback to $\TOGr_P(k,V)$ of the spinor bundle from $\OGr(k,V_P)$.
\end{itemize}
Note that all these bundles have rank $2^{g-k}$.
Restricting them to $X_k$, we obtain a family of vector bundles $\ti_P^*S_{P,k,\pm}$ (for non-branching $P$) and $\ti_P^*S_{P,k}$ (for branching $P$) on $X_k$.
In Proposition~\ref{proposition-construction-spinor} below we will show that they form a single vector bundle $\CS_k$ of rank $2^{g-k}$ on the product $X_k \times C$.
Meanwhile, we discuss some properties of these bundles that will be useful later.

\begin{lemma}\label{lemma-spinor-dual}
We have the following isomorphisms of sheaves on $\Gr(k,V)$:
\begin{equation*}
\begin{aligned}
\tj_{P*}(S_{P,k,\pm}^\vee) & \cong 
\tj_{P*}(S_{P,k,\pm})(-1), && \text{if $P$ is not branching and $g-k$ is odd,}\\
\tj_{P*}(S_{P,k,\pm}^\vee) & \cong 
\tj_{P*}(S_{P,k,\mp})(-1), && \text{if $P$ is not branching and $g-k$ is even,}\\
\tj_{P*}(S_{P,k}^\vee) & \cong \tj_{P*}(S_{P,k})(-1), && \text{if $P$ is branching.}
\end{aligned}
\end{equation*}
\end{lemma}
\begin{proof}
If $P$ is not branching, then $\TOGr_P(k,V)=\OGr(k,V)$ and by~\cite[Proposition~6.6]{KuznetsovIsotropic} the sheaf $S_{P,k,\pm}^\vee$ is isomorphic to~$S_{P,k,\pm}(-H)$ or~$S_{P,k,\mp}(-H)$
depending on the parity of $g - k$.
In the case when $P$ is branching it follows from \emph{loc.\,cit.}\/ that $S_{P,k}^\vee \cong S_{P,k}(-H') \cong S_{P,k}(E-H)$
(we use the notation from Proposition~\ref{proposition-togr}), hence, taking into account~\eqref{equation-ehh}, we get a short exact sequence
\begin{equation*}
0 \to S_{P,k}(-H) \to S_{P,k}^\vee \to S_{P,k}(E-H)|_{E} \to 0.
\end{equation*}
Note that $p_{1*}(S_{P,k}(E-H)|_E) = 0$.
Indeed, the fibers of the map $E = \OFl(k-1,k;V_P) \to \OGr(k-1,V_P)$ are quadrics,
and $S_{P,k}(E-H)$ restricts to each of these quadrics as the twist $S(-1)$ of the spinor bundle on the quadric,
which is an acyclic vector bundle.
Therefore, $p_{1*}(S_{P,k}(-H)) \cong p_{1*}(S_{P,k}^\vee)$.
Applying the functor $j_{P*}$ and taking into account the diagram~\eqref{diagram-maps-ti-tj}, we deduce the required isomorphism.
% The morphism $\tj_P$ factors through $p_1$, so let us push it forward to $\OGr_P(k,V)$.
% Note that the pushforward of $S_{P,k}(E-H)|_{E}$ to $\OGr(k-1,V_V)$ is zero since this bundle restricts to the fibers of the projection $E = \OFl(k-1,k;V_P) \to \OGr(k-1,V_P)$,
% which are quadrics, as $S(-1)$, which is an acyclic vector bundle on a quadric.
% Therefore, $p_{1*}(S_{P,k}(-H)) \cong p_{1*}(S_{P,k}^\vee)$. 
% Pushing this further to $\Gr(k,V)$, we deduce the required isomorphism.
% 
% The right term restricts to any fiber of $E = \OFl(k-1,k;V_P)$, which is a quadric, as $S(-1)$, hence it is acyclic and
% its pushforward is zero. The projection formula finishes the proof.
\end{proof}

Now we construct a global vector bundle on $X_k \times C$ whose restriction to the fibers of $X_k \times C$ over $C$ gives the bundles $S_{P,k\pm}$ and $S_{P,k}$ constructed above.
% We start with a useful Lemma.

% First, let $P \in \PP^1$ be a non-branching point. 
% Then $\TOGr_P(k,V) := \OGr_P(k,V)$ is a homogeneous space of the group
% $\Spin(N)$ and so by~\cite{KuznetsovIsotropic} 
% it carries a pair of vector bundles $S_{P,k,\pm}$ of rank $2^{N/2-1-k}$. 
% Similarly, if $P$ is a branching point, then we have a map $p_1:\TOGr_P(k,V) \to \OGr(k,V_P)$, 
% and the Grassmannian $\OGr(k,V_P)$ is a homogeneous space of the group $\Spin(N-1)$.
% So, again by~\cite{KuznetsovIsotropic} it carries one spinor bundle~$S$, also of rank $2^{N/2-k-1}$. Its pullback to $\TOGr_P(k,V)$
% is denoted by $S_{P,k}$. 

\begin{proposition}\label{proposition-construction-spinor}
There is a vector bundle $\CS_k$ 
% of rank $2^{N/2-k-1}$ 
on $X_k \times C$ such that for any point $P \in \PP^1$
\begin{itemize}
\item 
if $P$ is not a branching point and $f^{-1}(P) = \{x_+,x_-\} \subset C$, then 
\begin{equation}\label{eq:spinor-non-branching}
\ti_P^*S_{P,k,+} \cong \CS_{k,x_+} 
\qquad\text{and}\qquad
\ti_P^*S_{P,k,-} \cong \CS_{k,x_-};
\end{equation}
\item 
if $P$ is a branching point and $f^{-1}(P) = x \in C$, then
\begin{equation}\label{eq:spinor-branching}
\ti_P^*S_{P,k} \cong \CS_{k,x}.
\end{equation}
\end{itemize}
\end{proposition}
\begin{proof}
One could use the same construction as in Bondal--Orlov: consider a relative flag variety and push forward a spinor line bundle. 
However, for further compatibility we use another construction.
% Essentially, it consists of applying the relative version of~\cite[Section~6]{KuznetsovIsotropic}.
% For this, however, we need a family of nondegenerate quadrics with a section of the Lagrangian orthogonal Grassmannian at the generic point (in other words, the quadric at the generic point of the curve should be hyperbolic).
% So, we start by cooking up such a family.
% 
% First, we consider the family $q \colon \Sym^2V \otimes \CO_{\PP^1} \to \CO_{\PP^1}(1)$ of quadrics over $\PP^1$ corresponding to the pencil generated by $Q_0$ and $Q_\infty$.
% Its Lagrangian orthogonal Grassmannian at generic point lives over a degree~2 extension $\kk(C)/\kk(\PP^1)$ corresponding to adjoining the square root of $\det(q)$.
% Since the Brower group of the field $\kk(C)$ is trivial, the quadric $q$ considered as a quadric over $\kk(C)$ is hyperbolic.\footnote{Reference needed.}
% The family of quadrics $q$, however, is degenerate at some points of $C$.

% Next we get rid of degeneracy.
We consider the trivial vector bundle $\CV = V \otimes \CO_{\PP^1}$ and the family~$q$ of quadratic forms on $\CV$ with values in $\CO_{\PP^1}(1)$ generated by $q_0$ and $q_\infty$.
Denoting by $x_1,\dots,x_{2g+2} \in C$ the ramification points of the double covering $f \colon C \to \PP^1$ such that $f(x_i) = P_i$, we also consider a vector bundle
\begin{equation*}
\widehat\CV := \bigoplus_{i=1}^{2g+2} \CO_C(x_i)
\end{equation*}
on $C$, and the natural diagonal quadratic form
\begin{equation*}
\widehat{q} \colon \Sym^2(\widehat\CV) \twoheadrightarrow \bigoplus_{i=1}^{2g+2} \CO_C(2x_i) \to f^*\CO_{\PP^1}(1),
\end{equation*}
where the second map is induced by the isomorphisms
$\CO_C(2x_i) \cong f^*\CO_{\PP^1}(P_i) \cong f^*\CO_{\PP^1}(1)$.
It is easy to see that the composition of the quadratic form $\widehat{q}$ with the natural embedding $f^*\CV \to \widehat\CV$
% \begin{equation*}
% \eta \colon \CV \hookrightarrow \widehat\CV
% \end{equation*}
induces the form~$f^*q$ on~$f^*\CV$.
At the same time, the quadratic form $\widehat{q}$ is everywhere non-degenerate. 
Indeed, the map 
\begin{equation*}
\widehat\CV \xrightarrow{\ \widehat{q}\ } \widehat\CV^\vee \otimes f^*\CO_{\PP^1}(1)
\end{equation*}
is the direct sum of the maps $\CO_C(x_i) \to \CO_C(-x_i) \otimes f^*\CO_{\PP^1}(1)$, each of which is an isomorphism.

Consider the sheaf of even parts of Clifford algebras $\Clif_0$ on the curve $C$ associated with the family of quadratic forms $\widehat{q}$.
By~\cite[Proposition~3.13]{KuznetsovQuadrics} there is an \'etale double covering $\tilde{f} \colon \widetilde{C} \to C$ and a sheaf of Azumaya algebras~$\widetilde\Clif_0$ on $\widetilde{C}$
such that $\Clif_0 \cong \tilde{f}_*\widetilde\Clif_0$. 
Let us check that the double covering $\tilde{f}$ is trivial.
Indeed, 
% it is given by the Stein factorization for the natural projection $\OGr_C(g+1,\widehat{\CV}) \to C$ from the relative isotropic Grassmannian.
% On the other hand, 
over the complement $C_0 \subset C$ of the ramification divisor of $f$, the family of quadratic forms $\widehat{q}$ coincides with the pullback of the family $q$,
hence over $C_0$ the double covering~$\tilde{f}$ is isomorphic to the pullback of the double covering associated with the family $q$, which is nothing but $f \colon C \to \PP^1$.
% covering $f \colon C \to \PP^1$, i.e.,
In other words, we have an isomorphism
\begin{equation*}
\widetilde{C} \times_C C_0 \cong C_0 \times_{\PP^1} C_0.
\end{equation*}
The right hand side is the union of two connected components, hence $\widetilde{C}$ is the union of two irreducible components, hence $\tilde{f}$ is trivial.

The triviality of $\tilde{f}$ means that the sheaf of algebras $\Clif_0$ is isomorphic to a product of two Azumaya algebras 
swapped by the involution $\tau \colon C \to C$ of the double covering $f$.
Taking into account that the Brauer group of a curve $C$ is trivial (since the base field $\kk$ is algebraically closed of characteristic 0), we obtain an isomorphism
\begin{equation*}
\Clif_0 \cong \CEnd(\widehat\CS) \oplus \CEnd(\tau^*\widehat\CS)
\end{equation*}
for a vector bundle $\widehat\CS$ of rank $2^g$ on $C$, the \emph{spinor bundle} for the family of quadrics $\widehat{q}$.

Next, consider the relative orthogonal Grassmannian $\OGr_C(k,\widehat\CV)$ parameterizing $k$-dimensional subspaces in the fibers of $\widehat\CV$ 
isotropic with respect to the quadratic form $\widehat{q}$.
Let $\widehat\CU_k$ be the tautological bundle on $\OGr_C(k,\widehat\CV)$.
The Clifford multiplication induces a map 
\begin{equation*}
\widehat\CU_k \otimes \tau^*\widehat\CS \to \widehat\CS,
\end{equation*}
whose cokernel by~\cite[Lemma~6.1]{KuznetsovIsotropic} is the spinor bundle $\widehat\CS_k$ on $\OGr_C(k,\widehat\CV)$.
% hence $\OGr_{C_0}(g+1,\widehat{\CV})$ is isomorphic to the pullback of the relative isotropic Grassmannian for the family o
% 
% 
% Let us show that $\widehat{q}$ is \emph{neutral}, i.e., the vector bundle $\widehat\CV$ admits a direct sum decomposition 
% as a direct sum of two $\widehat{q}$-isotropic subbundles.
% For this, first, choose a pair of general points of the scheme $X_g$, such that the corresponding subspaces $U'_g,U''_g \subset V$ do not intersect.
% The natural map 
% \begin{equation*}
% (U'_g \oplus U''_g) \otimes \CO_{C} \to \CV \to \widehat\CV
% \end{equation*}
% is then an embedding of vector bundles

Finally, let $\CU_k$ be the tautological bundle on $X_k$ and consider on $X_k \times C$
% 
% Consider the relative orthogonal isotropic Grassmannian $\OGr_C(k,\widehat\CV)$ parameterizing $k$-dimensional subspaces in the fibers of $\widehat\CV$ 
% isotropic with respect to the quadratic form $\widehat{q}$.
% Let us check that 
the composition of the maps 
\begin{equation*}
\CU_k \boxtimes \CO_C \hookrightarrow V \otimes \CO_{X_k \times C} \cong \CO_{X_k} \boxtimes f^*\CV \hookrightarrow \CO_{X_k} \boxtimes \widehat\CV.
\end{equation*} 
It is an embedding of vector bundles.
% and that the image $\widehat{\CU}_k$ of $\CU_k \boxtimes \CO_C$ under this embedding  is isotropic with respect to $\widehat{q}$.
% Next we consider the subbindle $\CU_k \hookrightarrow V \otimes \CO_{X_k}$, pull it back to $X_k \times C$, and compose its embedding with the embedding $\CV \to \widehat\CV$.
% The composition $\CU_k \boxtimes \CO_C \hookrightarrow \CV \to \widehat\CV$ is everywhere monomorphic.
Indeed, the map~$f^*\CV \to \widehat\CV$ is monomorphic everywhere except over the points $x_i \in C$, where the kernels are spanned by the vectors~$v_{P_i} \in V$,
and as we have already mentioned several times, these vectors are not contained in the subspaces~$U_k \subset V$ isotropic for all quadrics~$Q_P$ in our pencil.
% We denote the obtained subbundle in~$\widehat\CV$ by~$\widehat\CU_k$.
Moreover, the image of $\CU_k \boxtimes \CO_C$ is isotropic with respect to $\widehat{q}$ over $C_0$, since so is $\CU_k$.
% Indeed, it is $\widehat{q}$-isotropic over~$C_0$ 
% % the complement of the ramification locus of $f \colon C \to \PP^1$ 
% since it coincides there with $\CU_k$.
It follows by continuity that it is everywhere isotropic.
Therefore, there is a map $X_k \times C \to \OGr_C(k,\widehat\CV)$ over $C$ such that $\CU_k \boxtimes \CO_C$ is isomorphic to the pullback of $\widehat\CU_k$ 
as a subbundle in $\CO_{X_k} \boxtimes \widehat\CV$.
% Since $\CU_k$ is an isotropic subbundle, it induces a morphism $X_k \times C \to \OGr_C(k,\widehat\CV)$.
% such that $\CU_k$ is the pullback of the tautological bundle.
We denote by~$\CS_k$ the pullback under this morphism of the spinor bundle $\widehat\CS_k$ on $\OGr_C(k,\widehat\CV)$.
% defined via the approach of~\cite[Section~6]{KuznetsovIsotropic}.
It remains to show that isomorphisms~\eqref{eq:spinor-non-branching} and~\eqref{eq:spinor-branching} hold.

When $P$ is not a branching point, the form $\widehat{q}$ over $P$ coincides with the form $q$. Hence, by a base change to points $x_\pm \in f^{-1}(P)$ we obtain the standard embedding 
\begin{equation*}
X_k \stackrel{i_P}\hookrightarrow \OGr_P(k,V) \cong \OGr(k,\widehat\CV_{x_\pm}) 
\end{equation*}
and~\eqref{eq:spinor-non-branching} follows from the compatibility of the construction of spinor bundlles
% ~\cite[Section~6]{KuznetsovIsotropic} 
with base changes.

When $P = P_i$ is a branching point, if $x = x_i$ is the corresponding ramification point,  we have a commutative diagram
\begin{equation*}
\xymatrix{
V \ar@{->>}[r] \ar[d]^q & 
V_P \ar@{^{(}->}[r] \ar[d]^{q'_P} & 
\widehat\CV_{x} \ar[d]^{\widehat{q}} 
\\
V^\vee & 
V_P^\vee \ar@{_{(}->}[l] 
& \widehat\CV^\vee_{x} \ar@{->>}[l]
}
\end{equation*}
that shows that by a base change to $x_i$ we obtain the composition
\begin{equation*}
\xymatrix@1{X_k \ar[r]^-{i_P} &
\OGr_P(k,V) \ar@{-->}[rr]^-{p_2\circ p_1^{-1}} &&
\OGr(k,V_P) \ar[r] & 
\OGr(k,\widehat\CV_x).
}
\end{equation*}
The composition of the first two arrows here is $p_2 \circ p_1^{-1} \circ i_P = p_2 \circ \ti_P$, so it is enough to note that the restriction of the spinor bundle from $\OGr(k,\widehat\CV_x)$ to $\OGr(k,V_P)$ is the spinor bundle.
Indeed, this is one of the standard properties of spinor bundles, since the odd-dimensional quadric $Q_P \subset \PP(V_P)$ is a smooth hyperplane section of the smooth even-dimensional quadtic $\widehat{Q}_x \subset \PP(\widehat\CV_x)$.
This finally proves~\eqref{eq:spinor-branching}.
\end{proof}

\begin{remark}
The family of quadratic forms $\widehat{q}$ on $\widehat\CV$ was also used extensively in~\cite{DR}.
\end{remark}

In the particular case when $k = g - 1$ Desale and Ramanan proved in~\cite{DR} that $X_{g-1}$ is isomorphic 
to the moduli space $M_C(2,L)$ of rank 2 vector bundles on $C$ with the determinant equal to $L$, a fixed line bundle of odd degree. 
We show that the universal bundle for the moduli problem is isomorphic to the spinor bundle $\CS_{g-1}$ constructed above up to a twist.
Note that by Lemma~\ref{lemma:pic-x-c} every line bundle on~$X_{g-1} \times C$ is the tensor product of (pullbacks of) a line bundle on $X_{g-1}$ and a line bundle on $C$.
% (all twists are ``by a line bundle on $X_{g-1} \times C$'', and Lemma~\ref{lemma:pic-x-c} implies that it
% is the same as ``by a line bundle on $X_{g-1}$ and a line bundle on $C$'').

\begin{proposition}\label{proposition-spinor-universal}
The vector bundle $\CS_{g-1}$ on $X_{g-1}\times C$ is isomorphic up to a twist to the universal rank~$2$ bundle on the product $M_C(2,L) \times C$.
\end{proposition}
\begin{proof}
We use the construction of $\CS_{g-1}$ from the proof of Proposition~\ref{proposition-construction-spinor}.
According to this construction, the bundle $\CS_{g-1}$ is the pullback of the spinor bundle on $\OGr_C(g-1,\widehat\CV)$.
If $\widehat\CU_{g-1}$ and $\widehat\CS_{g-1}$ are the tautological and the spinor bundle on $\OGr_C(g-1,\widehat\CV)$
respectively and $\tau$ is its involution induced by the hyperelliptic involution of $C$, we have an isomorphism
\begin{equation*}
\widehat\CS_{g-1} \otimes \tau^*\widehat\CS_{g-1} \cong (\widehat\CU_{g-1}^\perp / \widehat\CU_{g-1}) \otimes \det(\widehat\CU_{g-1})^\vee
\end{equation*}
(the map is induced by the action of the odd part of the Clifford algebra on the half-spinor modules, 
and the fact that it is an isomorphism can be verified at geometric points, in which case it follows from representation theory of the group $\Spin(2g+2)$).
Pulling it back to $X_{g-1} \times C$, we obtain an isomorphism
\begin{equation*}
\CS_{g-1} \otimes \tau^*\CS_{g-1} \cong (\CU_{g-1}^\perp /\CU_{g-1}) \otimes \det(\CU_{g-1})^\vee.
\end{equation*}

On the other hand, let $\CE$ be the universal bundle on the product $M_C(2,L) \times C$
and consider the pullback of $\CU_{g-1}$ to $M_C(2,L) \times C$ via the isomorphism $X_{g-1} \cong M_C(2,L)$.
By~\cite[Proposition~5.9]{DR} there is an isomorphism up to a twist between $\CU_{g-1}^\perp/\CU_{g-1}$ and the tensor product $\CE \otimes \tau^*(\CE)$,
and, moreover, the rank~2 bundle $\CE$ is unique (up to a twist) with this property.
Combining these two observations, we conclude that the spinor bundle $\CS_{g-1}$ is isomorphic to~$\CE$ up to a twist.
\end{proof}

\begin{remark}\label{remark-spinor-poincare}
   A similar argument shows that under the isomorphism $X_g \cong \Pic^0(C)$ proved in~\cite{MilesReid} and~\cite{DR}
the spinor bundle $\CS_g$ is isomorphic to a twist of the Poincar\'e bundle.
\end{remark}

\subsection{Full faithfulness in the hyperelliptic case}

For each $1 \le k \leq g-1$ consider the Fourier--Mukai functor
defined by the spinor bundle $\CS_k$ constructed in Proposition~\ref{proposition-construction-spinor}
\begin{equation*}
\Phi_k = \Phi_{\CS_k} \colon \BD(C) \to \BD(X_k).
\end{equation*}
When $k = 1$ this functor coincides with the functor of Bondal and Orlov~\cite{BondalOrlov}.
Our main result, fully faithfulness of the functor $\Phi_k$, is thus a generalization of their result~\cite[Theorem~2.7]{BondalOrlov}.

\begin{theorem}\label{theorem:hyperelliptic}
Let $C$ be a hyperelliptic curve of genus $g \ge 2$ and let $k$ be an integer such that~$k \le g-1$.
Then the functor $\Phi_k:\BD(C) \to \BD(X_k)$ is fully faithful.
\end{theorem}

Along this section we fix $k$ and omit it from the notation.
In particular, we abbreviate $\Phi_k$ to $\Phi$ and~$\CS_k$ to $\CS$.
In order to prove that $\Phi$ is fully faithful we use the following criterion by Bondal and Orlov.

\begin{theorem}[\protect{\cite[Theorem~1.1]{BondalOrlov}}]\label{theorem:bondal-orlov}
Let $C$ and $X$ be smooth projective varieties and $\CE \in \BD(C\times X)$.
The Fourier--Mukai functor $\Phi_\CE$ is fully faithful if and only if
\begin{equation*}
\begin{aligned}
(i)\  & \Ext^i(\Phi_\CE(\CO_{x_1}), \Phi_\CE(\CO_{x_2})) = 0, && \text{if $x_1 \ne x_2$},\\
(ii)\ & \Hom(\Phi_\CE(\CO_{x}), \Phi_\CE(\CO_{x})) = \kk, \quad \Ext^i(\Phi_\CE(\CO_{x}), \Phi_\CE(\CO_{x})) = 0 && \text{for $i < 0$ and $i > \dim C$},
\end{aligned}
\end{equation*}
%   \begin{enumerate}
%     \item[i)] $\Ext^i(\Phi_\CE(\CO_{x_1}), \Phi_\CE(\CO_{x_2}))=0 \qquad \text{ for all } x_1\neq x_2$,
%     \item[ii)] $\Hom(\Phi_\CE(\CO_{x}), \Phi_\CE(\CO_{x}))=\kk,\quad
%                \Ext^i(\Phi_\CE(\CO_{x}), \Phi_\CE(\CO_{x}))=0 \qquad \text{ for } i\notin [0,\ldots,\dim C]$,
%   \end{enumerate}
for all points $x,x_1,x_2\in C$, where $\CO_x$ denotes the skyscraper sheaf.
\end{theorem}

We are going to check that the conditions of the criterion are satisfied for $\CE = \CS$.
To be more precise, we reduce the necessary verifications to some cohomology computations that are performed in the next section
(see Corollary~\ref{corollary-vanishing}, Lemma~\ref{lemma-same-p-type-d} and Lemma~\ref{lemma-same-p-type-b}).
In Proposition~\ref{proposition-ext-different-p} we discuss the case when~$f(x_1) \ne f(x_2)$, and in Proposition~\ref{proposition-ext-same-p} the case when $f(x_1) = f(x_2)$.
In both cases we check that the sufficient conditions of Theorem~\ref{theorem:bondal-orlov} hold true, so Theorem~\ref{theorem:hyperelliptic} follows.

% These computations, in turn, are performed in the next section (see Corollary~\ref{corollary-vanishing},
% Lemma~\ref{lemma-same-p-type-d} and Lemma~\ref{lemma-same-p-type-b}).

First, we compute the $\Ext$ groups for points $x_1$, $x_2$ in different fibers of $C$ over $\PP^1$.

\begin{proposition}\label{proposition-ext-different-p}
Let $x_1,x_2 \in C$ and assume that $f(x_1) \ne f(x_2)$.
% For $P_1 \ne P_2$ we have 
Then $\Ext^\bullet(\CS_{x_1},\CS_{x_2}) = 0$.
\end{proposition}
\begin{proof}
Consider the diagram
(we abbreviate $\ti_{P_1}$, $\ti_{P_2}$, $\tj_{P_1}$, and $\tj_{P_2}$ to $\ti_1$, $\ti_2$, $\tj_1$, and $\tj_2$ respectively.)
\begin{equation*}
\xymatrix{
X \ar[r]^-{\ti_1} \ar[d]_{\ti_2} & \TOGr_{P_1}(k,V) \ar[d]^{\tj_1} \\
\TOGr_{P_2}(k,V) \ar[r]^-{\tj_2} & \Gr(k,V)
}
\end{equation*}
Recall that $\CS_{x_1} = \ti_{1}^*S_{P_1}$ and $\CS_{x_2} = \ti_{2}^*S_{P_2}$ by Proposition~\ref{proposition-construction-spinor}, where 
% $P_1 = f(x_1)$, $P_2 = f(x_2)$, 
$P_i = f(x_i)$, and for each of the two points $P \in \{P_1,P_2\}$ if $P$ is not a branching point of $f$,
then $S_P$ is one of the two spinor bundles on~$\TOGr_P(k,V) = \OGr_P(k,V)$, while if $P$ is a branching point,
then $S_P$ is the pullback via $p_2$ of the spinor bundle on $\OGr(k,V_P)$.
Then by adjunction we have
\begin{equation*}
\Ext^\bullet(\CS_{x_1},\CS_{x_2}) =
\Ext^\bullet(\ti_{1}^* S_{P_1}, \ti_{2}^* S_{P_2}) =
\Ext^\bullet(S_{P_1}, \ti_{1*}\ti_{2}^* S_{P_2}) =
H^\bullet(\TOGr_{P_1}(k,V),S_{P_1}^\vee \otimes \ti_{1*}\ti_{2}^* S_{P_2}).
\end{equation*}
The diagram is Cartesian and $\Tor$-independent by Lemma~\ref{lemma-tor-independence}, so we have a base change isomorphism \cite[Theorem~3.10.3]{Lipman}
\begin{equation*}
\ti_{1*} \ti_2^* S_{P_2} \cong \tj_1^* \tj_{2*} S_{P_2}.
\end{equation*}
We can now  rewrite
\begin{multline*}
H^\bullet(\TOGr_{P_1}(k,V),S_{P_1}^\vee \otimes \ti_{1*}\ti_{2}^* S_{P_2}) =
H^\bullet(\TOGr_{P_1}(k,V),S_{P_1}^\vee \otimes \tj_1^* \tj_{2*} S_{P_2}) = \\
H^\bullet(\Gr(k,V),\tj_{1*}(S_{P_1}^\vee \otimes \tj_1^* \tj_{2*} S_{P_2})) = 
H^\bullet(\Gr(k,V),\tj_{1*}(S_{P_1}^\vee) \otimes \tj_{2*}(S_{P_2})).
\end{multline*}
By Lemma~\ref{lemma-spinor-dual}, we have $\tj_{1*}(S_{P_1}^\vee) \cong \tj_{1*}(S_{P_1}')(-1)$, where $S'_{P_1}$ is one of the spinor bundles corresponding to the same point $P_1$,
so the right hand side can be rewritten as $H^\bullet(\Gr(k,V),\tj_{1*}(S'_{P_1}) \otimes \tj_{2*}(S_{P_2}) \otimes \CO(-1))$.
By Corollary~\ref{corollary-vanishing} the latter is zero, which finishes the proof.
\end{proof}

Now consider the case when the points are in the same fiber of $C$ over $\PP^1$,
in particular, they might coincide.

\begin{proposition}\label{proposition-ext-same-p}
Let $x_1,x_2 \in C$ and assume that $f(x_1) = f(x_2)$. Then $\Ext^\bullet(\CS_{x_1},\CS_{x_2}) = 0$ in case $x_1 \ne x_2$,
and $\Ext^\bullet(\CS_{x_1},\CS_{x_2}) = \kk \oplus \kk[-1]$ in case $x_1 = x_2$.
\end{proposition}
\begin{proof}
Let $P = f(x_1) = f(x_2)$. 
Then we can assume that one of the following possibilities hold:
\begin{enumerate}
\item 
$\CS_{x_1} = \ti_P^*S_{P,+}$ and $\CS_{x_2} = \ti_P^*S_{P,-}$ if $x_1 \ne x_2$;
\item 
$\CS_{x_1} = \CS_{x_2} = \ti_P^*S_{P,+}$ if $x_1 = x_2$ and $P$ is not branching;
\item 
$\CS_{x_1} = \CS_{x_2} = \ti_P^*S_P$ if $x_1 = x_2$ and $P$ is branching.
\end{enumerate}
% Assume the first. 
First, assume (1) holds.
Then
\begin{equation*}
\Ext^\bullet(\CS_{x_1},\CS_{x_2}) = 
\Ext^\bullet(\ti_P^*S_{P,+},\ti_P^*S_{P,-}) = 
\Ext^\bullet(S_{P,+},\ti_{P*}\ti_P^*S_{P,-}) = 
\Ext^\bullet(S_{P,+},S_{P,-} \otimes \ti_{P*}\CO_{X_k}).
\end{equation*}
Since the scheme $X_k$ is the zero locus of a regular section of $\Sym^2\CU_k^\vee$ on $\OGr_P(k,V)$ (see Lemma~\ref{lemma-tor-independence}), 
its structure sheaf has a Koszul resolution with terms of the form $\Lambda^m(\Sym^2\CU_k)$. 
So, it is enough to show that
% \begin{equation*}
$\Ext^\bullet(S_{P,+},S_{P,-} \otimes \Lambda^m(\Sym^2\CU_k)) = 0$
% \end{equation*}
for all $m \ge 0$. This is the statement of Lemma~\ref{lemma-same-p-type-d}~(1).

Next, assume (2) holds.
% Now assume the second. 
Then the same argument reduces the claim of the proposition to the computation of~$\Ext^\bullet(S_{P,+},S_{P,+} \otimes \Lambda^m(\Sym^2\CU_k))$,
% \end{equation*}
and Lemma~\ref{lemma-same-p-type-d}~(2) shows that the only nontrivial terms are
\begin{equation*}
\Hom(S_{P,+},S_{P,+})=\kk\quad\text{and}\quad \Ext^2(S_{P,+},S_{P,+}\otimes \Lambda^1(\Sym^2\CU_k)) = \kk,
\end{equation*}
and a spectral sequence argument shows that $\Ext^\bullet(\CS_{x_1},\CS_{x_2}) = \kk \oplus \kk[-1]$.

Finally, assume (3) holds.
% Finally, assume the third. 
Then we need to compute on $\TOGr_P(k,V)$ the groups
% \begin{equation*}
\begin{multline*}
\Ext^\bullet(S_{P},S_{P} \otimes \Lambda^m(\Sym^2\CU_k)) =
\Ext^\bullet(p_2^*S,p_2^*S \otimes \Lambda^m(\Sym^2\CU_k)) \\ \cong
\Ext^\bullet(S,p_{2*}(p_2^*S \otimes \Lambda^m(\Sym^2\CU_k))) \cong
\Ext^\bullet(S,S \otimes p_{2*}(\Lambda^m(\Sym^2\CU_k))),
\end{multline*}
% \end{equation*}
% Since the spinor bundle is a pullback from $\OGr(k,V_P)$, it is enough to compute
% \begin{equation*}
% \Ext^\bullet(S,S \otimes p_{2*}(\Lambda^m(\Sym^2\CU_k))).
% \end{equation*}
and Lemma~\ref{lemma-same-p-type-b} in conjunction with a typical spectral sequence argument finish the proof.
\end{proof}

\begin{remark}
Theorem~\ref{theorem:hyperelliptic} has the following corollary.
For each point $x \in C$ the bundle $\CS_x \cong \Phi_k(\CO_x)$ on $X_k$ is simple, 
and the map $\Ext^1(\CO_x,\CO_x) \to \Ext^1(\CS_x,\CS_x)$ induced by the functor $\Phi_k$ is an isomorphism.
Therefore $\Phi_k$ induces an \'etale map $x \mapsto \CS_x$ from the curve $C$ to the moduli stack of simple objects in the derived category $\BD(X_k)$.
Since also $\Hom(\CS_{x_1},\CS_{x_2}) = 0$ for $x_1 \ne x_2$, this map is bijective.
If one could check that all the bundles $\CS_x$ are stable, this would identify the curve $C$ with a component of the moduli space of stable bundles on $X_k$.
This is, indeed, can be checked in case $k = g - 1$ of rank 2 bundles, but for smaller $k$ (and higher ranks) this becomes hard.
\end{remark}

\subsection{From hyperelliptic to generic curves}

In this section we show that full faithfulness for hyperelliptic curves implies the same for generic curves.
It follows from a more general result.

Assume $B$ is a smooth base scheme, $X \to B$ and $Y \to B$ are schemes over $B$ that we assume for simplicity to be smooth and proper over $B$,
and $\CE \in \BD(X\times_B Y)$ is an object. For each point $b \in B$ let $\CE_b$ be the (derived) pullback 
of $\CE$ to $X_b\times Y_b \subset X\times_B Y$, and consider the corresponding Fourier--Mukai functor 
$\Phi_{\CE_b}:\BD(Y_b) \to \BD(X_b)$. Define the {\sf locus of full faithfulness} of $\CE$ to be the subset
\begin{equation*}
\FFF(\CE) := \{ b \in B\ |\ \text{$\Phi_{\CE_b}$ is fully faithful} \} \subset B.
\end{equation*}

\begin{proposition}\label{proposition-fff-open}
If $p:X \to B$ and $q:Y \to B$ are smooth projective over a smooth $B$, then the locus of full faithfulness $\FFF(\CE) \subset B$ is open.
\end{proposition}
\begin{proof}
Consider the object $\CE^! := \CE^\vee \otimes \omega_{Y/B}[\dim Y/B] \in \BD(X\times_B Y)$.
Then for each $b \in B$ the functor $\Phi_{\CE^!_b}:\BD(X_b) \to \BD(Y_b)$ is right adjoint to $\Phi_{\CE_b}$.
Further, let $\CE^!\circ \CE \in \BD(Y\times_B Y)$ be the convolution of the two kernels. 
Then for each $b \in B$ the functor $\Phi_{(\CE^!\circ\CE)_b}:\BD(Y_b) \to \BD(Y_b)$ is isomorphic 
to the composition of functors $\Phi_{\CE^!_b} \circ \Phi_{\CE_b}$.
Finally, there is a canonical morphism 
\begin{equation*}
\Delta_*\CO_Y \to \CE^! \circ \CE
\end{equation*}
such that for each $b \in B$ the induced morphism $\id_{Y_b} = \Phi_{\Delta_*\CO_{Y_b}} \to \Phi_{(\CE^!\circ\CE)_b}$ 
is the unit of the adjunction. 

It is well known that a functor is fully faithful if and only if the unit of its adjunction is an isomorphism.
It follows that $\FFF(\CE)$ is the complement of the projection to $B$ of the support of the cone of the morphism $\Delta_*\CO_Y \to \CE^! \circ \CE$. 
The support of an object of $\BD(Y\times_B Y)$ is closed in $Y \times_B Y$, hence proper over $B$, hence its image in $B$ is closed. 
Thus, its complement $\FFF(\CE)$ is open in $B$.
\end{proof}

We apply this proposition to prove Theorem~\ref{theorem-main}.

\begin{proof}[Proof of Theorem~\textup{\ref{theorem-main}}]
Consider a category $\FB_{g,n}^{r,d}$ fibered in groupoids over the category of $\kk$-schemes defined by associating to a scheme $S$ the category of collections
% \begin{equation*}
$\FB_{g,n}^{r,d}(S) = \{(\CC,x_1,\dots,x_n,\CE)\}$,
% \end{equation*}
where 
\begin{itemize}
\item 
$\CC \to S$ is a smooth projective curve of genus $g$,
\item 
$x_1,\dots,x_n \colon S \to \CC$ are non-intersecting sections of $\CC \to S$,
% \item 
% $\CL$ is a line bundle on~$\CC$ of degree $d$ on all fibers of $\CC$ over $S$, and 
\item 
$\CE$ is a vector bundle of rank $r$ on $\CC$ such that for any geometric point $s \in S$
the restriction $\CE_s$ of~$\CE$ to the fiber $\CC_s$ of $\CC$ over the point $s$ is stable of degree~$d$.
% (we could choose any other odd degree here).
\end{itemize}
A morphism from $(\CC,x_1,\dots,x_n,\CE)$ over~$S$ to $(\CC',x'_1,\dots,x'_n,\CE')$ over $S'$ is 
\begin{itemize}
\item 
a morphism $\varphi \colon S \to S'$,
\item  
an isomorphism $\varphi_\CC \colon \CC \xrightarrow{\ \sim\ } \CC' \times_{S'} S$ such that $x_i = x'_i \times_{S'} S$ for all $1 \le i \le n$,
\item 
a line bundle $\CL \in \Pic^0(\CC/S)$, and
\item 
an isomorphism $\varphi_\CE \colon \CE \cong \varphi_\CC^*(\CE') \otimes \CL$.
\end{itemize}
It is easy to see that $\FB_{g,n}^{r,d}$ is a stack over the category of schemes in fppf topology (the stack of $n$-pointed genus $g$ curves with a stable vector bundle of rank $r$ and degree $d$).
Below we only need the cases $n \le 1$ and $r \le 2$, and we will assume that $d$ is odd.
Of course, the stability condition is only relevant when~$r > 1$.

Consider the following Cartesian diagram of morphisms of stacks
\begin{equation*}
\xymatrix{
& \FB_{g,1}^{2,d} \ar[dr]^{\det} \ar[dl]_{\mathrm{forget}}
\\
\FB_{g,0}^{2,d} \ar[dr]_{\det} 
&& 
\FB_{g,1}^{1,d} \ar[dl]^{\mathrm{forget}}
\\
& \FB_{g,0}^{1,d} \ar[d]^\pi
\\
& \CM_g
% \ar[dr]_{c}
% && \CM_{g,1} \ar[dl]
% \\
% && \CM_g,
}
\end{equation*}
where the south-east arrows take $(\CC,x_1,\dots,x_n,\CE)$ to $(\CC,x_1,\dots,x_n,\det(\CE))$
(for $n = 0$ and $n = 1$),
the south-west arrows forget the point~$x_1$, $\CM_g$ is the stack of genus $g$ curves, and the vertical arrow $\pi$ forgets the line bundle.

Assume $d$ is odd. 
Consider the universal rank~2 vector bundle $\CE$ on the stack $\FB_{g,1}^{2,d}$ 
(here we view the stack~$\FB_{g,1}^{2,d}$ as the universal curve over the stack~$\FB_{g,0}^{2,d}$),
and the Fourier--Mukai functor 
\begin{equation*}
\Phi_\CE \colon \BD(\FB_{g,0}^{2,d}) \to \BD(\FB_{g,1}^{1,d}).
\end{equation*}
Clearly, over a point $(C,L)$ of $\FB_{g,0}^{1,d}$ the top part of the diagram becomes
\begin{equation*}
\xymatrix{
& M_C(2,L) \times C \ar[dl] \ar[dr] \\
M_C(2,L) && C
}
\end{equation*}
and $\CE_{C,L}$ becomes the universal vector bundle on $M_C(2,L) \times C$. 
Therefore, by Theorem~\ref{theorem:hyperelliptic} the functor~$\Phi_{\CE_{C,L}}$ is fully faithful as soon as the curve $C$ is hyperelliptic.
Thus, the open (by Proposition~\ref{proposition-fff-open}) substack $\FFF(\CE) \subset \FB_{g,0}^{1,d}$ contains the preimage under $\pi$ of the hyperelliptic locus $\CM_g^{\mathrm{he}} \subset \CM_g$.
Let us check that $\FFF(\CE)$ is the preimage under $\pi$ of an open substack of $\CM_g$.

Since the morphism $\pi$ is proper, it is enough to show that if $L_1$ and~$L_2$ are two different line bundles of degree $d$ on a curve~$C$, the functor $\Phi_{\CE_{C,L_1}}$ is fully faithful if and only if the functor~$\Phi_{\CE_{C,L_2}}$ is.
Indeed, the moduli spaces $M_C(2,L_1)$ and $M_C(2,L_2)$ are isomorphic via the map 
\begin{equation*}
E \mapsto E \otimes (L_1^{-1} \otimes L_2)^{1/2}
\end{equation*}
determined by a choice of a square root of the degree zero line bundle $L_1^{-1} \otimes L_2$ on $C$.
This isomorphism identifies the universal vector bundles $\CE_{C,L_1}$ and $\CE_{C,L_2}$ (up to a twist), 
hence identifies the functors $\Phi_{\CE_{C,L_1}}$ and $\Phi_{\CE_{C,L_2}}$ (up to a composition with a twist on $M_C(2,L)$ and a twist on $C$).
Since a line bundle twist is an autoequivalence, the first is an autoequivalence if and only if the second is.

Thus, we checked that there is an open substack in the moduli stack $\CM_g$ containing the hyperelliptic locus $\CM_g^{\mathrm{he}}$ whose preimage in $\FB_{g,0}^{1,d}$ equals $\FFF(\CE)$.
This means that for any curve from this open substack the functor $\Phi_{\CE_{C,L}} \colon \BD(C) \to \BD(M_C(2,L))$ is fully faithful for any line bundle of degree $d$.
It is also clear that if $d'$ is another odd number, the open substacks associated with $d$ and $d'$ coincide.
This completes the proof of Theorem~\ref{theorem-main}.
\end{proof}

% \newpage

\section{Cohomological computations}\label{section-bbw}

In this section we carry out the cohomological computations on which the proofs of the previous section rely.

\subsection{Preliminaries}\label{ss:preliminaries}

We start with introducing necessary notation.

Let $\You{w}{h}$ denote the set of Young diagrams inscribed in a rectangle of width~$w$ and height $h$.
We think of such a Young diagram $\alpha$ as of an integer sequence $(\alpha_1,\ldots,\alpha_h)$ such that
\begin{equation*}
w\geq\alpha_1\geq\alpha_2\geq\ldots\geq \alpha_h\geq 0.
\end{equation*}
As usual, we denote by $\alpha^T \in \You{h}{w}$ the transposed diagram.
The number of boxes in $\alpha$ is denoted by~$|\alpha|$ and the length of its diagonal by $\dlen(\alpha)$:
\begin{equation*}
|\alpha| = \sum \alpha_i,
\qquad\text{and}\qquad
% \alpha^T_i = \max \left\{j \mid \alpha_j\geq i \right\}.
\dlen(\alpha) = \max \left\{ i \mid \alpha_i\geq i \right\}.
\end{equation*}
Both do not change under transposition.

Let $s = \dlen(\alpha)$. The diagrams
\begin{equation*}
\topd\alpha := (\alpha_1-s, \dots,\alpha_s-s) 
\qquad\text{and}\qquad 
\btmd\alpha = (\alpha_{s+1},\dots,\alpha_h)
\end{equation*}
are called the {\sf head} and the {\sf tail} of $\alpha$ respectively, see Figure~\ref{fig:head-tail}.
\begin{figure}[h]
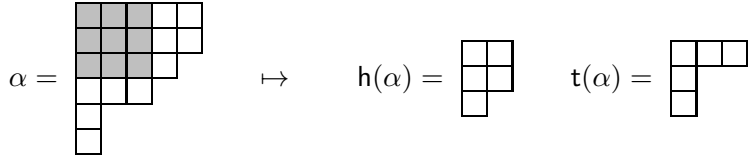

\ytableausetup{centertableaux, smalltableaux}
$\alpha = {}$
\ydiagram[*(lightgray)]
  {3,3,3}
*[*(white)]{5,5,4,3,1,1} 
$\qquad\mapsto\qquad$
$\topd\alpha = {}$
\ydiagram[*(white)]{2,2,1} 
\qquad
$\btmd\alpha = {}$
\ydiagram[*(white)]{3,1,1} 
\caption{The head and the tail}\label{fig:head-tail}
\end{figure}
Transposition interchanges and transposes the head and the tail of a diagram.

A diagram $\alpha$ is called {\sf symmetric} if $\alpha^T = \alpha$. This is equivalent to $\btmd\alpha = \topd\alpha^T$.

Let $\alpha$ be a Young diagram and $s = \dlen(\alpha)$. 
We define the {\sf horizontal $p$-expansion} of $\alpha$ as
\begin{equation*}
\expd{p}{0}\alpha := (\alpha_1 + p, \dots, \alpha_s + p, \alpha_{s+1}, \dots, \alpha_h)
\end{equation*}
and the {\sf vertical $q$-expansion} of $\alpha$ as
\begin{equation*}
\expd{0}{q}\alpha := (\alpha_1, \dots, \alpha_s, \underbrace{s, \dots, s}_{\text{$q$ times}}, \alpha_{s+1}, \dots, \alpha_h)
\end{equation*}
The transposition of a horizontal expansion is a vertical expansion of the transposition. 

Given a Young diagram $\alpha$, we denote by $\Sigma^\alpha$ the associated Schur functor, see, e.g., \cite[Section 2.1]{WeymanBBW}, where, however, a different notation and convention is used.
According to our convention, for $\alpha = (a)$ we have $\Sigma^\alpha\CU = \Sym^a\CU$ and $\Sigma^{\alpha^T}\CU = \Lambda^a\CU$.

\subsection{Cohomology computations on orthogonal Grassmannians}

Let $V$ be a vector space of dimension $N$ with a non-degenerate quadratic form.
We denote by $\Spin(V)$ the corresponding spin group.
It is a simple simply connected group of type $B_n$ if $N = 2n+1$ is odd, and of type $D_n$ if $N = 2n$ is even.

For every $k \le n-1$ we consider the orthogonal Grassmannian $\OGr(k,V)$.
It is a homogeneous variety of $\Spin(V)$.
It comes with the tautological vector bundle $\CU$ of rank $k$ (the restriction of the tautological bundle from $\Gr(k,V)$).
Clearly, $\CU$ is a $\Spin(V)$-equivariant vector bundle, and so is $\Sigma^\alpha\CU$ for any Young diagram $\alpha \in \You{w}{k}$.

On the other hand, $\OGr(k,V)$ carries some spinor bundles, see~\cite[Section~6]{KuznetsovIsotropic}.
In type $B_n$ 
% ($\veps = 1/2$) 
there is one spinor bundle $S$, and in type $D_n$ 
% ($\veps = 0$) 
there are two of them: $S_+$ and $S_-$.
The spinor bundles are $\Spin(V)$-equivariant, the spaces of their global sections 
\begin{equation*}
\BS = H^0(\OGr(k,V),S),
\qquad
\BS_\pm = H^0(\OGr(k,V),S_\pm),
\end{equation*}
are representations of $\Spin(V)$ called the half-spinor representations.

In order to treat both cases simultaneously, we write
\begin{equation*}
N = 2(n + \veps),
\qquad\qquad 
\veps \in \left\{0,1/2\right\},
\end{equation*}
so that $\veps = 0$ corresponds to the even case (type $D_n$) and $\veps = 1/2$ corresponds to the odd case (type~$B_n$).
For the same reason
% In order to treat both even and odd cases simultaneously, 
we will sometimes denote $S_+$ simply by $S$ in type $D_n$, and understand both~$S_+$ and~$S_-$ as $S$ in type $B_n$.
The same convention is applied to half-spinor representations.

\begin{lemma}\label{lemma-cohomology-ualpha}
Assume $N \ge 2k + 2$ and let $\beta\in\You{N-k}{k}$ be a Young diagram.
Then the vector bundle $\Sigma^\beta\CU \otimes S$ is acyclic unless $\beta$ is a horizontal expansion of a symmetric Young diagram $\nu \in \You{k}{k}$:
\begin{equation*}
\beta = \expd{N-2k}0\nu, \qquad \nu^T = \nu.
\end{equation*}
Moreover, if $\nu$ is symmetric and $s = \dlen(\nu)$, then
\begin{equation*}
H^p(\OGr(k,V), \Sigma^{\expd{N-2k}0\nu}\CU\otimes S) = 
\begin{cases}
\BS_{(-1)^s}, & \text{for $p = s(N-2k) + (|\nu| - s)/2$,} \\
% \text{for $p = |\beta| - (|\nu|+s)/2$,} \\
0, & \text{otherwise.}
\end{cases}
\end{equation*}
\end{lemma}
\begin{proof}
The bundle $\Sigma^\beta\CU \otimes S$ is $\Spin(V)$-equivariant, so its cohomology can be computed via the Borel--Bott--Weil Theorem~\cite[Corollaries~4.3.7, 4.3.9]{WeymanBBW}
in terms of the Weyl group action on the weight lattice 
\begin{equation*}
\Lambda := \Lambda(B_n) = \Lambda(D_n) = \ZZ^n + ( (1/2,1/2,\dots,1/2) + \ZZ^n ),
\end{equation*}
of the group $\Spin(V)$.
The weight corresponding to this vector bundle is
% Let us use the Borel--Bott--Weil theorem.
% 	Recall that the vector bundle $\Sigma^\beta\CU \otimes S$ corresponds to the weight
\begin{equation*}
\delta := \lambda\left(\Sigma^\beta\CU\otimes S\right)
 = \left(-\beta_k+\hlf,\ -\beta_{k-1}+\hlf,\ \ldots,\ -\beta_1+\hlf;\ \hlf,\ \ldots,\ \hlf \right) \in \Lambda
\end{equation*}
(we use a semicolon to divide the first $k$ coordinates from the last $n-k$).
% Next, we should add to $\delta$ 
The special weight $\rho$ (the sum of all fundamental weights, or, equivalently, the hals-sum of the positive roots) for the group $\Spin(V)$ can be uniformly written as
\begin{equation*}
\rho = (n-1+\veps, \dots, n-k+\veps;\ n-k-1+\veps, \dots, \veps) \in \Lambda.
\end{equation*}
% The resulting weight is
Summing it up with $\delta$, we obtain
\begin{equation}\label{eq:mu}
\mu := \delta + \rho = \left((n-1)+\veps+\hlf-\beta_k,\dots,(n-k)+\veps+\hlf-\beta_1;\ (n-k-1)+\veps+\hlf,\dots,\veps+\hlf\right).
\end{equation}
The Weyl group acts on the weight lattice by permutations of coordinates and changes of their signs (in type $D_n$ only an even number of sign changes is allowed).
By the Borel--Bott--Weil Theorem the bundle is acyclic if this weight is singular (fixed by an element of the Weyl group), i.e., if the absolute values of two coordinates are the same or if one coordinate is zero (in type $B_n$).

Observe that all terms of $\mu$ are integer if $\veps = \hlf$ (i.e.\ in type $B_n$) and half-integer if $\veps = 0$ (i.e.\ in type~$D_n$).
Moreover, the absolute values of $\mu_i$ are bounded from above by $n -\hlf + \veps$.
It easily follows that $\mu$ is singular unless $\mu = \sigma\lambda_\pm$, where
\begin{equation}\label{eq:mu:sigma}
\lambda_\pm := \left((n-1)+\veps+\hlf, \ldots, (n-k)+\veps+\hlf;\ (n-k-1) + \veps + \hlf,\dots,\veps \pm \hlf\right)
\end{equation}
and $\sigma$ is an element of the Weyl group (the minus sign in the last entry is allowed in type $D_n$ only).

Comparing the weights~\eqref{eq:mu} and~\eqref{eq:mu:sigma}, we see that their last $n-k$ coordinates coincide (up to a sign in the last one). 
Hence, $\sigma$ permutes and changes signs of the first $k$ terms only (and possibly changes the sign of the last term in type $D_n$, note however,
that the total number of sign changes in type $D_n$ is always even, so what happens with the last term is determined by the other terms uniquely). 
Moreover, the first $k$ terms in~\eqref{eq:mu} strictly decrease. This means that $\sigma$ changes the signs of some subset of the first~$k$ terms
of~\eqref{eq:mu:sigma}, invert their order, and put at the end of the first $k$ terms. Explicitly, this means that there is
a decomposition $\{1,\dots,k\} = I \sqcup J = \{i_1,\dots,i_r\} \sqcup \{j_1,\dots, j_s\}$ with $r + s = k$, such that
\begin{align*}
 n - 1 - \beta_{s+r} + \veps + \hlf &= n-i_1 + \veps + \hlf, 
 &&
 \dots, &
 n - r - \beta_{s+1} + \veps + \hlf &= n-i_r + \veps + \hlf,\\
 n - r - 1 - \beta_{s} + \veps + \hlf & = -n + j_s - \veps - \hlf, && \dots, &
 n - k - \beta_1 + \veps + \hlf & = -n + j_1 - \veps - \hlf.
\end{align*}
Taking into account that $2n + 2\veps = N$ we can rewrite this as
\begin{equation}\label{eq:beta}
\begin{aligned}
&\beta_1 = (N-k) + 1 - j_1, 
&& \dots, & 
&\beta_s = (N-k) + s - j_s,
\\
&\beta_{s+1} = i_r - r, 
&&\dots, & 
&\beta_{s+r} = i_1 - 1.
%  \beta_{s+r} &= i_1 - 1, &&\dots, & \beta_{s+1} &= i_r - r,\\
%  \beta_s &= (N-k) + s - j_s, &&\dots, & \beta_1 &= (N-k) + 1 - j_1.
\end{aligned}
\end{equation} 
In other words,
\begin{equation*}
\beta = \expd{N-2k}0\nu,
\qquad 
\nu = (k+1-j_1,\dots,k+s-j_s,i_r-r,\dots,i_1-1)
\end{equation*}
(indeed, since $\nu_s = s + k - j_s \ge s$ and $\nu_{s+1} = i_r - r \le k - r = s$, we see that $\dlen(\nu) = s$).
Thus the cohomology is zero unless $\beta$ is horizontally expanded.
Furthermore, we have 
\begin{equation}\label{eq:head-tail-nu}
\topd\nu = (r+1-j_1,\dots,r+s-j_s) \in \You{r}{s}
\qquad\text{and}\qquad
\btmd\nu = (i_r-r,\dots,i_1-1) \in \You{s}{r},
\end{equation} 
and the condition that $I$ and $J$ are complementary subsets of $\{1,\dots,k\}$ means that $\topd\nu = \btmd\nu^T$.
In other words, $\nu$ is symmetric.

To summarize, we have proved that the cohomology is zero, unless $\beta$ is the horizontal expansion of a symmetric Young diagram $\nu \in \You{k}{k}$,
and in the latter case the cohomology is isomorphic to the representation of $\Spin(V)$ with highest weight $\lambda_\pm$, i.e., to $\BS_\pm$,
% $\BS_{(-1)^s}$ (since $\sigma^{-1}(\mu) - \rho$ is the weight corresponding to a half-spinor representation)
and sits in the degree equal to the length of $\sigma$ in the Weyl group.

It remains to compute the length of $\sigma$ and to understand the sign of $\BS$. 
First consider the case of type~$B_n$ (so that we have $N = 2n + 1$). 
Then $\sigma$ is the following composition of simple reflections. 
First we put the term at position $j_s$ to the last position ($n-j_s$ simple reflections), change its sign (1 reflection), and then put it to position $k$ ($n-k$~reflections). 
Then we do the same composition of simple reflections for $j_{s-1}$, \dots, $j_1$. 
Thus,
\begin{equation*}
 \ell(\sigma) = (2n + 1 - k -j_s) + \dots + (2n + 1 - k - j_1) = s(N - k) - \sum j_t.
\end{equation*}
On the other hand, in type $D_n$ (when $N = 2n$) the last simple reflection changes the signs of the last two terms and interchange them. Hence $\sigma$ is the following composition.
First, we put the term at position~$j_s$ to the last position ($n-j_s$ simple reflections), change the sign and the position of the last two terms (1~reflection), 
and then put the previous to the last term to position $k$ ($n-1-k$ reflections). 
Then we do the same composition of simple reflections for $j_{s-1}$, \dots, $j_1$. Thus,
\begin{equation*}
 \ell(\sigma) = (2n - k -j_s) + \dots + (2n - k - j_1) = s(N - k) - \sum j_t.
\end{equation*}
On the other hand, since $\nu$ is symmetric with $\dlen(\nu) = s$ and $\topd\nu$ is given by~\eqref{eq:head-tail-nu}, we have
\begin{equation*}
|\nu| = s^2 + 2|\topd\nu| = s^2 + 2(rs + s(s+1)/2 - \sum j_t).
\end{equation*}
Taking into account $r = k - s$, we deduce $\ell(\sigma) = s(N-2k) + (|\nu|-s)/2$.

Finally, in the case of $D_n$ the simple reflections listed above include $s$ copies of the $n$-th reflection, hence in the end the sign of the last coordinate is $(-1)^s$.
Thus, the resulting representation is $\BS_{(-1)^s}$.
\end{proof}

\begin{lemma}\label{lemma-wt}
Assume $N \ge 2k + 3$, and let $S$ and $S'$ be two \textup{(}possibly isomorphic\textup{)} spinor bundles on~$\OGr(k,V)$.
Let $\beta$ be a Young diagram of width $w(\beta)$ and height $h(\beta) \le k$ such that 
\begin{equation}\label{eq:w-h}
w(\beta) \le h(\beta) + 1.
\end{equation} 
The vector bundle $\Sigma^\beta\CU \otimes S^\vee \otimes S'$ is acyclic unless $\beta = (t)$ for some $t \le 2$
and $S' = S_{(-1)^t}$.
% and the types of the spinor bundles are the same when $t$ is even and opposite when $t$ is odd.
Moreover, if $t \le 2$ we have
\begin{equation*}
H^\bullet(\OGr(k,V),\Sym^t\CU \otimes S_+^\vee \otimes S_{(-1)^t}) = H^\bullet(\OGr(k,V),\Sym^t\CU \otimes S_-^\vee \otimes S_{(-1)^{t+1}}) = \kk[-t].
\end{equation*}
\end{lemma}
\begin{proof}
The argument is completely analogous to the one we used to prove Lemma~\ref{lemma-cohomology-ualpha}.
The only new ingredient is the decomposition of the tensor product of the dual spinor and the spinor bundles into a sum of irreducible equivariant $\Spin(V)$-bundles,
for which we use representation theory of the spin-group.
Using it, we conclude that $\Sigma^\beta\CU \otimes S^\vee \otimes S'$ is a direct sum of $\Spin(V)$-equvariant vector bundles corresponding to some weights of the form
\begin{equation}\label{eq:delta-t}
\delta_t := (-\beta_k,\dots,-\beta_1;\ \underbrace{1,\dots,1}_{\text{$t$ times}},0,\dots,0)
% \qquad 
% \delta_{n-k}^- = (-\beta_k,\dots,-\beta_1;\ \underbrace{1,\dots,1}_{\text{$n-k-1$ times}},-1),}
\end{equation}
for $0 \le t \le n-k$.
The weight $\mu := \delta_t + \rho$ can be written as
\begin{equation*}
\mu = ((n-1) + \veps - \beta_k,\dots,(n-k) + \veps - \beta_1;\ (n-k) + \veps, \dots, (n-k-t+1) + \veps, (n-k-t-1) + \veps, \dots, \veps).
\end{equation*}
The absolute values of all entries are bounded by $(n-1) + \veps$, hence either~$\mu$ is singular or $\mu = \sigma\rho$
% \begin{equation*}
% \mu = \sigma\rho.
% \end{equation*}
for some $\sigma$ in the Weyl group. 
The same argument as in Lemma~\ref{lemma-cohomology-ualpha} shows that there is a decomposition $\{1,\dots,k-1,k+t\} = I \sqcup J$ and $\beta$ is given by equations~\eqref{eq:beta}.
% with the same formulas for $\beta$.

First assume $I \ne \{1,\dots,r\}$ for some $r$. 
Let $m+1$ be the first gap in $I$, so that $i_{m} = m$ and $i_{m+1} \ge m+2$.
By~\eqref{eq:beta} we have $h(\beta) = s + (r-m) = k - m$ and $w(\beta) = N - k + 1 - j_1$.
Hence, the assumption~\eqref{eq:w-h} implies
% $w(\beta) \le h(\beta) + 1$ implies
% , hence $(N-k) + 1 - j_1 = w(\beta) \le h(\beta) + 1 =  k - m + 1$ implies 
$j_1 \ge (N - 2k) + m \ge m + 2$. 
Thus $m + 1$ is not in~$I \sqcup J$.
This is only possible if $J = \emptyset$.
In this case $\beta = (t)$ (hence, $\Sigma^\beta\CU = \Sym^t\CU$), and then condition~\eqref{eq:w-h} implies $t \le 2$.

Now assume $I = \{1,\dots,r\}$ and $s > 1$.
% Then we have $\beta_{s-1} = (N-k) + (s-1) - (k-1) = (N-2k) + s > 0$, hence the height of $\beta$ is either $s$, or $s-1$ (if $\beta_s = (N-2k) + (s-t) = 0$). 
By~\eqref{eq:beta} we have $h(\beta) \le s$.
% Then the height of $\beta$ is at most $s$.
On the other hand, 
% the width of $\beta$ is 
\begin{equation*}
w(\beta) = (N-k) + 1 - j_1 = (N-k) + 1 - (r+1) = (N-2k) + s \ge s+2.
\end{equation*}
This contradicts with the assumption $w(\beta) \le h(\beta) + 1$, hence this case is impossible.

Finally assume that $I = \{1,\dots,k-1\}$ and $J = \{k+t\}$. 
By~\eqref{eq:beta} we have $\beta = (N - 2k + 1 - t)$.
The assumption~\eqref{eq:w-h} implies $t \ge N - 2k - 1$.
On the other hand, we have $t \le n - k$ by~\eqref{eq:delta-t}.
We conclude that $n - k \ge 2n + 2\veps - 2k - 1$, hence $k \ge n + 2\veps - 1$.
This contradicts to $N \ge 2k + 3$, so this case is also impossible.

Summarizing, we see that $t \le 2$ and $\sigma$ moves the $k$-th entry $t$ steps to the right.
% , or (in type $D_n$ only if $k = n - 1$) moves the $k$-th entry one step to the right and then changes the signs of the last two entries.
Taking also in account that the weights $\delta_0$ and $\delta_2$ appear in the decomposition of $\Sym^t\CU \otimes S^\vee \otimes S'$ only when $S' \cong S$,
while $\delta_1$ appears only when $S' \cong S_-$, we deduce the claim.
% 
% % (otherwise $\beta = (N -2k + 1 - t)$ and the above argument applies).
% 
% If $t = 0$, then $\sigma = 1$, hence the cohomology is $\kk$, and this happens only when the spinor bundles are the same.
% If $t = 1$, then $\sigma$ moves the $k$-th entry one steps to the right, hence $\ell(\sigma) = 1$, the cohomology is $\kk[-1]$, and this happens only when the spinor bundles are opposite.
% Finally, if $t = 2$ then $\sigma$ moves the $k$-th entry two steps to the right, hence $\ell(\sigma) = 2$, the cohomology is $\kk[-1]$, and this happens only when the spinor bundles are the same.
\end{proof}

\subsection{The resolution}

In this section we apply the above cohomology computations to show that the pushforwards of the spinor bundles under the map $j:\OGr(k,V) \to \Gr(k,V)$ have very nice locally free resolutions.
Recall that $N = \dim V$, and $S$ is the spinor bundle on $\OGr(k,V)$ when $N$ is odd or $S = S_+$ is one of the two spinor bundles if $N$ is even. 
Recall also the Young diagram notation introduced in Section~\ref{ss:preliminaries}.

\begin{proposition}\label{lemma-resolution}
The sheaf $j_*S$ has a resolution by equivariant vector bundles on $\Gr(k,V)$ of the form
\begin{equation*}
0\to \CF_{k(k+1)/2} \to \ldots \to \CF_1 \to \CF_0\to j_*S\to 0,
\end{equation*}
with its $t$-th term equal to
\begin{equation*}
\CF_t = \bigoplus_{\substack{\nu \in \You{k}{k},\ \nu^T = \nu \\ |\nu|+\dlen(\nu) = 2t}}
\left(\BS_{(-1)^{\dlen(\nu)}} \otimes \Sigma^{\expd0{N-2k}\nu}\CU^\perp\right).
\end{equation*}
\end{proposition}

\begin{proof}
Recall that, by Kapranov's results~\cite{Kapranov}, the bounded derived category
$\BD(\Gr(k,V))$ has a full exceptional collection consisting of equivariant vector bundles of the form
\begin{equation}\label{eq:kapranov}
\BD(\Gr(k,V)) = \left<\Sigma^\alpha\CU^\perp \mid \alpha\in\You{k}{N-k}\right>.
\end{equation}
Its left dual exceptional collection is $\left<\Sigma^\beta\CU^\vee \mid \beta\in\You{N-k}{k}\right>$,
and for any $\CG \in\BD(\Gr(k,V))$ there is a spectral sequence with the first page
\begin{equation*}
E^{pq}_1=\bigoplus_{p=-|\alpha|} \Ext^q(\Sigma^{\alpha^T}\CU^\vee, \CG)\otimes \Sigma^\alpha\CU^\perp \quad\Longrightarrow\quad H^{p+q}(\CG).
\end{equation*}
Consider this spectral sequence for $\CG = j_*S$. Then
\begin{equation*}
\Ext^q(\Sigma^{\alpha^T}\CU^\vee, j_*S) = \Ext^q(j^*\Sigma^{\alpha^T}\CU^\vee, S) = H^q(\OGr(k,V), \Sigma^{\alpha^T}\CU\otimes S).
\end{equation*}
We see that the nontrivial terms in the spectral sequence come from those $\alpha \in \You{k}{N-k}$
for which we have $H^\bullet(\OGr(k,V), \Sigma^{\alpha^T}\CU\otimes S)\neq 0$.
Lemma~\ref{lemma-cohomology-ualpha} tells us that that the latter happens if and only if $\alpha^T = \expd{N-2k}{0}\nu$, where $\nu$ is symmetric, i.e.\ if $\alpha = \expd0{N-2k}\nu$. 
For such $\alpha$ the only nontrivial cohomology is placed in degree $q = s(N-2k) + (|\nu|-s)/2$, where $s  =\dlen(\nu)$ and
equals $\BS_{(-1)^s}$. Finally, remark that $p = -|\alpha| = -s(N-2k) - |\nu|$, thus $p + q = -(|\nu| + s)/2$.
\end{proof}

\begin{example}
For $k = 1$ the resolution looks very simple:
\begin{equation*}
0 \to \BS_- \otimes \CO(-1) \to \BS \otimes \CO \to j_*S \to 0
\end{equation*}
(we use here the fact that $\Sigma^{1^{N-1}}\CU^\perp = \det\CU^\perp = \CO(-1)$).
% and put $\BS_- = \BS_{-1}$ to make the notation more compact).
This is just the standard resolution of the pushforward of the spinor bundle from 
a quadric to the ambient projective space.

For $k = 2$ the resolution is of the form
\begin{equation*}
0 \to 
\BS \otimes \CO(-2) \to 
\BS_- \otimes \CU^\perp \otimes \CO(-1) \to 
\BS_- \otimes (V/\CU) \otimes \CO(-1) \to 
\BS \otimes \CO \to j_*S \to 0,
\end{equation*}
and for $k = 3$ it looks like
\begin{multline*}
0 \to 
\BS_- \otimes \CO(-3) \to 
\BS \otimes \Lambda^2\CU^\perp \otimes \CO(-2) \to 
\BS \otimes \operatorname{ad}(\CU^\perp) \otimes \CO(-2) \to \\ \to
(\BS_- \otimes \Sym^2(\CU^\perp) \otimes \CO(-1) \ \oplus\  \BS \otimes \Sym^2(V/\CU) \otimes \CO(-2)) \to \\ \to
\BS_- \otimes \operatorname{ad}(\CU^\perp) \otimes \CO(-1) \to 
\BS_- \otimes \Lambda^2(V/\CU) \otimes \CO(-1) \to 
\BS\otimes\CO \to S \to 0,
\end{multline*}
where $\operatorname{ad} = \Sigma^{1,0,\dots,0,-1}$ is the Schur functor of the adjoint representation.
\end{example}

The above resolution is very useful for handling the spinor bundles coming from orthogonal Grassmannians of nondegenerate quadrics. 
As for degenerate quadrics, we have a weaker version of the above result which is still enough for our purposes.
We now assume that $V$ is even-dimensional and use the notation and conventions of Sections~\ref{ss:ogr} and~\ref{ss:spinor}.
In particular, $S_P$ denotes a spinor bundle on $\TOGr_P(k,V)$.

\begin{lemma}\label{lemma-spinor-subcat}
% Let $S$ be a spinor bundle on $\TOGr_P(k,V)$. 
The sheaf $\tj_{P*}S_P$ is contained in the subcategory of $\BD(\Gr(k,V))$
generated by the bundles $\Sigma^{\expd0{N-2k}\nu}\CU_k^\perp$, where $\nu$ runs through the set of all Young diagrams in~$\You{k}{k}$
such that $|\topd\nu| \le |\btmd\nu|$:
\begin{equation*}
\tj_{P*}S_P \in \langle \Sigma^{\expd0{N-2k}\nu}\CU_k^\perp \mid \nu \in \You{k}{k},\ |\topd{\nu}|\leq |\btmd{\nu}| \rangle.
\end{equation*}
\end{lemma}
\begin{proof}
When $P$ is not branching, $\TOGr_P(k,V)\simeq \OGr(k,V)$ and the result follows from Proposition~\ref{lemma-resolution}, as $\topd\nu = (\btmd\nu)^T$ in this case.
	
Now assume that $P$ is branching. Consider the diagram
\begin{equation*}
\xymatrix{
& & \TOGr_P(k,V) \ar[dl]_{p_1} \ar[dr]^{p_2} \ar@/_/[dll]_{\tj_P} \\
\Gr(k,V) & \OGr_P(k,V) \ar[l]^{j_P} & & \OGr(k,V_P) \\
}
\end{equation*}
As in the proof of Proposition~\ref{lemma-resolution}, we decompose $\tj_{P*}S_P = j_{P*}p_{1*}p_2^*S$ with respect to the exceptional collection~\eqref{eq:kapranov}.
% Kapranov's exceptional collection $\BD(\Gr(k,V)) = \langle\Sigma^\alpha\CU_k^\perp\rangle$. 
In order to do that, it is enough to compute $H^\bullet(\Gr(k,V),\Sigma^\beta\CU_k \otimes j_{P*}p_{1*}p_2^*S)$ for every $\beta \in \You{N-k}{k}$. 
Applying the projection formula, one gets
\begin{equation*}
\Sigma^\beta\CU_k \otimes j_{P*}p_{1*}p_2^*S \simeq j_{P*}p_{1*}(\Sigma^\beta\CU_k \otimes p_2^*S),
\end{equation*}
so it is enough to compute $H^\bullet(\TOGr_P(k,V),\Sigma^\beta\CU_k \otimes p_2^*S)$. 
Pushing forward to $\OGr(k,V_P)$ and using the projection formula once again, we reduce everything to the computation of $H^\bullet(\OGr(k,V_P),p_{2*}(\Sigma^\beta\CU_k) \otimes S)$.
	
Recall that $p_2$ is the natural projection from $\TOGr_P(k,V)=\Gr_{\OGr(k,V_P)}(k,\CO \oplus \CU_k')$.
% and set $s = \dlen(\beta)$. 
Using a relative version of the Borel--Bott--Weil theorem (\cite[Corollary~4.1.9]{WeymanBBW}), we find
\begin{equation}\label{eq:p2-star}
p_{2*}(\Sigma^\beta\CU_k) \cong
\begin{cases}
\Sigma^{\expd01\tau}(\CO \oplus \CU'_k)[-\dlen(\tau)], & \text{if $\beta = \expd10\tau$, $\tau \in \You kk$,}\\
0, & \text{otherwise.}
\end{cases}
% p_{2*}(\Sigma^\beta\CU_k) = 
% \begin{cases}
% 0, & \text{if $\beta_s = s$,}\\
% \Sigma^{(\beta_1-1,\dots,\beta_s-1,s,\beta_{s+1},\dots,\beta_k)}(\CO \oplus \CU_k')[-s], & \text{otherwise}.
% \end{cases}
\end{equation}
Assume we are in the first case and denote $\beta' = \expd01\tau = (\beta_1-1,\dots,\beta_s-1,s,\beta_{s+1},\dots,\beta_k)$, where $\beta = \expd10\tau$ and $s = \dlen(\tau)$.
Remark that $\beta'\in\You{N'-k}{k+1}$, where $N'=\dim V_P = N-1$.
	
Recall also the standard decomposition \cite[Proposition~2.3.1]{WeymanBBW}
\begin{equation}\label{eq:gelfand-zeitlin}
\Sigma^{\beta'}(\CO \oplus \CU_k') \cong \bigoplus_{\beta'_1 \geq \gamma_1 \geq \beta'_2 \geq \ldots \geq \beta'_k \geq \gamma_k \geq \beta'_{k+1}} \Sigma^\gamma\CU_k'.
\end{equation}
% % \Sigma^{\beta'}(\CO \oplus \CU_k')\simeq\bigoplus \Sigma^\gamma\CU_k'$, where $\gamma$ runs over the
% set of all diagrams $\gamma\in\You{N'-k}{k}$ such that
% \begin{equation}\label{eq:lsub:skew}
% \beta'_{i+1}\leq\gamma_i\leq\beta'_i\qquad\text{for $i=1,\ldots,k$}.
% \end{equation}
From Lemma~\ref{lemma-cohomology-ualpha} we know that $H^\bullet(\OGr(k,V_P),\Sigma^\gamma\CU_k' \otimes S) \neq 0$ if and only if $\gamma = \expd{N'-2k}0{\mu}$ with symmetric $\mu$.
Assuming this, we deduce from inequalities in~\eqref{eq:gelfand-zeitlin} that $s := \dlen(\beta)=\dlen(\beta')=\dlen(\gamma)$ and $\beta'_s \geq \gamma_s \ge (N' - 2k) + \mu_s \ge (N' - 2k) + s$. 
Therefore, $\beta_s = \beta'_s + 1 \ge (N - 2k) + s$.
Since $s = \dlen(\beta)$, it follows that $\beta = \expd{N-2k}0{\nu'}$ for some ${\nu'} \in \You{k}{k}$ with $\dlen({\nu'}) = s$.

Next we relate the number of boxes in the head and the tail of ${\nu'}$.
On one hand, we have
\begin{equation*}
|\btmd{\nu'}| = \beta_{s+1}+\ldots+\beta_k = \beta'_{s+2}+\ldots+\beta'_{k+1} \le \gamma_{s+1}+\ldots+\gamma_k = |\btmd{\mu}|,
% |\btmd{\mu}| = \gamma_{s+1}+\ldots+\gamma_k\geq \beta'_{s+2}+\ldots+\beta'_{k+1} = \beta_{s+1}+\ldots+\beta_k = |\btmd{\nu'}|,
\end{equation*}
and on the other hand
\begin{equation*}
|\topd{{\nu'}}| =
\beta_{1} + \ldots + \beta_s - s(N - 2k + s) = 
\beta'_{1} + \ldots + \beta'_{s} - s(N' - 2k + s) \ge
\gamma_{1} + \ldots + \gamma_{s} - s(N' - 2k + s) =
|\topd{\mu}|.
% |\topd{\mu}| = 
% \gamma_{1} + \ldots + \gamma_{s} - s(N' - 2k + s) \leq 
% \beta'_{1} + \ldots + \beta'_{s} - s(N' - 2k + s) = 
% \beta_{1} + \ldots + \beta_s - s(N - 2k + s) = 
% |\topd{{\nu'}}|.
\end{equation*}
Since $\mu$ is symmetric, we deduce that $|\btmd{{\nu'}}| \leq |\btmd{\mu}| = |\topd{\mu}| \leq |\topd{{\nu'}}|$.
	
We have just checked that $H^\bullet(\Gr(k,V),\Sigma^{\beta}\CU_k \otimes \tj_{P*}S_P) \neq 0$ only if $\beta = \expd{N-2k}0{\nu'}$, where $\nu' \in \You{k}{k}$ with $|\btmd{\nu'}| \le |\topd{\nu'}|$.
We conclude that in the decomposition of $\tj_{P*}S_P$ we have only those $\Sigma^\alpha\CU_k^\perp$ for which $\alpha = \expd0{N-2k}{\nu}$, where $\nu \in \You{k}{k}$ with $|\btmd{\nu}| \ge |\topd{\nu}|$.
\end{proof}

\subsection{Vanishing}

Now, finally, we use the resolutions constructed in the previous section to deduce the required vanishing.

\begin{lemma}\label{lemma-vanishing-terms}
Assume $N = \dim V \ge 2k + 2$.
For $\mu, \nu \in \You {k}{k}$ such that either
\begin{itemize}
\item 
$\dlen(\mu)\neq \dlen(\nu)$, or
\item 
$\dlen(\mu) = \dlen(\nu)$ and  $|\topd{\mu}|+|\topd{\nu}|\leq |\btmd{\mu}| + |\btmd{\nu}|$, 
\end{itemize}
we have
\begin{equation*}
H^\bullet(\Gr(k,V),\Sigma^{\expd0{N-2k}\mu}\CU^\perp \otimes \Sigma^{\expd0{N-2k}\nu}\CU^\perp \otimes \CO(-1)) = 0.
\end{equation*}
\end{lemma}

\begin{proof}
Consider the decomposition
\begin{equation*}
\Sigma^{\expd0{N-2k}\mu}\CU^\perp \otimes \Sigma^{\expd0{N-2k}\nu}\CU^\perp = \bigoplus \Sigma^\gamma\CU^\perp,
\end{equation*}
where $\gamma$ runs over a set of Young diagrams of height $N - k$ prescribed by the Littlewood--Richardson rule.
To prove the required vanishing it is enough to show that $H^\bullet(\Gr(k,V),\Sigma^\gamma\CU^\perp(-1)) = 0$ for all such~$\gamma$.
The Borel--Bott--Weil theorem implies that the latter happens if and only if the weight
\begin{equation*}
\tau = (N-1, N-2, \ldots, N-k;\, N-k + \gamma_1, \dots, 2+\gamma_{N-k-1}, 1+\gamma_{N-k})
\end{equation*}
of $\GL(N)$ is fixed by an element of the Weyl group (which acts by permuting coordinates), i.e.\ if $\tau$ has two equal coordinates.
	
Let us denote $r=\dlen(\mu)$ and $s=\dlen(\nu)$ and set $\hat{\mu} = \expd0{N-2k}\mu$, $\hat\nu = \expd0{N-2k}\nu$.
By definition
\begin{equation*}
\hat\mu = (\mu_1, \ldots, \mu_r, \underbrace{r, \ldots, r}_{N-2k}, \mu_{r+1},\ldots, \mu_k)
\quad\text{and}\quad
\hat\nu = (\nu_1, \ldots, \nu_s, \underbrace{s, \ldots, s}_{N-2k}, \nu_{s+1},\ldots, \nu_k).
\end{equation*}
It follows directly from the Littlewood--Richardson rule and inequalities $\mu_1,\nu_1 \leq k$ (which are due to the fact that $\mu,\nu \in \You k{k}$) that
\begin{equation}\label{eq:lvt:ineq}
\hat\mu_i \leq \gamma_i \leq \hat\mu_{i} + k, \qquad
\hat\nu_i \leq \gamma_i \leq \hat\nu_i + k.
\end{equation}
Now, for $i = r+1$ the first inequality becomes $r \leq \gamma_{r+1} \leq r+k$ and for $i = s+1$ the second transforms into $s \leq \gamma_{s+1} \leq s+k$.
For the corresponding terms $\tau_{k+r+1} = N-k-r+\gamma_{r+1}$ and $\tau_{k+s+1} = N-k-s+\gamma_{s+1}$ we conclude
\begin{equation*}
N-k \leq \tau_{k+r+1}, \tau_{k+s+1} \leq N.
\end{equation*}
Recall that the first $k$ terms of $\tau$ occupy the range $[N-k,\ldots,N-1]$. 
Now, whenever $r \neq s$, i.e.\ $\dlen(\mu) \ne \dlen(\nu)$, either $\tau_{k+r+1} = \tau_{k+s+1}$ or one of these numbers coincides with one of the first $k$ coordinates of $\tau$.
In both cases the weight is singular, hence the cohomology vanishes.
	
We are left with the case $r = s$, $\tau_{k+s+1} = N$, hence $\gamma_{s+1}=k+s$.
Let us look at the next term $\tau_{k+s+2}=N-k-s-1+\gamma_{s+2}$. 
The second inequality in~\eqref{eq:lvt:ineq} implies that $s \le \gamma_{s+2} \le k + s$ (we use here the fact that $N - 2k \ge 2$), hence $N-k-1 \leq \tau_{k+s+2} \leq N-1$, 
and we deduce that either $\tau_{k+s+2}$ coincides with one of the first $k$ terms (and we are done) or $\tau_{k+s+2}=N-k-1$ and $\gamma_{s+2}=s$.
Again, in the first case the weight is singular, hence the cohomology vanishes.

So, assume that we are in the latter case. Thus, $\gamma$ is of the form
\begin{equation*}
\gamma = (\gamma_1,\ldots,\gamma_s,k+s,\underbrace{s,\ldots,s}_{N-2k+1},\gamma_{N-2k+s+1},\ldots,\gamma_{N-k}).
\end{equation*}
It follows from the Littlewood--Richardson rule that
\begin{align*}
\gamma_1 + \ldots + \gamma_{s+1} \leq \hat\mu_1+\ldots+\hat\mu_{s+1} + \hat\nu_1+\ldots+\hat\nu_{s+1} = |\topd{\mu}| + |\topd{\nu}| + 2s(s+1)
\end{align*}
Meanwhile, $\gamma_1\geq\gamma_2\geq\ldots\geq\gamma_s\geq\gamma_{s+1}=s+k$, which implies $\gamma_1+\ldots+\gamma_{s+1}\geq (s+1)(s+k)$ and hence
\begin{equation}\label{eq:lvt:ineq3}
|\topd{\mu}| + |\topd{\nu}| \geq (s+1)(k-s).
\end{equation}
Similarly, we have inequalities
\begin{align*}
\gamma_{N-2k+s+1} + \ldots + \gamma_{N-k} \geq \hat\mu_{N-2k+s+1} + \ldots + \hat\mu_{N-k} + \hat\nu_{N-2k+s+1} + \ldots + \hat\nu_{N-k} =
|\btmd{\mu}| + |\btmd{\nu}|
\end{align*}
and so inequalities $s=\gamma_{N-2k+s}\geq\gamma_{N-2k+s+1}\geq\ldots\geq \gamma_{N-k}$ imply that
\begin{equation}\label{eq:lvt:ineq4}
|\btmd{\mu}| + |\btmd{\nu}| \leq s(k-s).
\end{equation}
Combining inequalities~\eqref{eq:lvt:ineq3} and~\eqref{eq:lvt:ineq4} with the assumption of the lemma, we deduce that $s=k$, hence $\mu = \nu = (k,\dots,k)$, and
\begin{equation*}
\Sigma^{\expd0{N-2k}\mu}\CU^\perp \simeq \Sigma^{\expd0{N-2k}\nu}\CU^\perp \simeq \CO(-k).
\end{equation*}
The desired vanishing follows from $2k+1<N$.
\end{proof}

Now we switch back to the notation introduced in Sections~\ref{ss:ogr} and~\ref{ss:spinor}.

\begin{corollary}\label{corollary-vanishing}
For any non-branching points $P_1 \ne P_2$ we have 
\begin{equation*}
H^\bullet(\Gr(k,V),\tj_{P_1*}(S_{P_1}) \otimes \tj_{P_2*}(S_{P_2}) \otimes \CO(-1)) = 0.
\end{equation*}
\end{corollary}

\begin{proof}
From Lemma~\ref{lemma-spinor-subcat}  we know that
\begin{equation*}
\tj_{P_i*}(S_{P_i})\in \langle \Sigma^{\expd0{N-2k}\nu}\CU_k^\perp \mid \nu \in \You{k}{k},\ |\topd{\nu}|\leq |\btmd{\nu}| \rangle.
\end{equation*}
for both $i = 1$ and $i = 2$.
The vanishing now follows from Lemma~\ref{lemma-vanishing-terms}.
\end{proof}

\begin{lemma}\label{lemma-same-p-type-d}
Assume $\dim V = 2n$, $k\leq n-2$. Then on $\OGr(k, V)$ one has
\begin{enumerate}
\item
$\displaystyle\Ext^\bullet(S_+,S_- \otimes \Lambda^\bullet(\Sym^2\CU)) = \Ext^\bullet(S_-,S_+ \otimes \Lambda^\bullet(\Sym^2\CU)) = 0$,
\item
$\displaystyle\Ext^\bullet(S_+,S_+ \otimes \Lambda^m(\Sym^2\CU)) \cong \Ext^\bullet(S_-,S_- \otimes \Lambda^m(\Sym^2\CU)) \cong 
% \kk \oplus \kk[-2]$.}
\begin{cases}
\kk[-2m], & \text{for $m=0,1$,} \\
0, & \text{otherwise.}
\end{cases}$
\end{enumerate}
\end{lemma}
\begin{proof}
As in Lemma~\ref{lemma-wt}, let $S$ and $S'$ be two (possibly isomorphic) spinor bundles on $\OGr(k,V)$.
Let us compute
% \begin{equation*}
$\Ext^\bullet(S,S' \otimes \Lambda^\bullet(\Sym^2\CU)) = H^\bullet(\OGr(k,V), S^\vee\otimes S' \otimes \Lambda^\bullet(\Sym^2\CU))$.
% \end{equation*}
Recall that (\cite[Proposition~2.3.9]{WeymanBBW})
\begin{equation}\label{eq:lm-s2}
\Lambda^m(\Sym^2\CU) = \bigoplus_{\substack{\nu \in \You kk,\ \nu^T = \nu,\\|\nu| + \dlen(\nu) = 2m}} \Sigma^{\expd10\nu}\CU,
\end{equation}
and by Lemma~\ref{lemma-wt} we have $H^\bullet(\OGr(k,V), \Sigma^{\expd10\nu}\CU \otimes S^\vee\otimes S') = 0$ unless $\nu = (0)$ or $\nu = (1)$.
Moreover, in the first case we have $m = 0$, and the cohomology is nontrivial (and equals $\kk$) only if $S' \cong S$.
Similarly, in the second case we have $m = 1$, the cohomology is nontrivial (and equals $\kk[-2]$) only if $S' \cong S$.
% only if the spinor bundles agree, and then equals $\kk[-2]$.
% 
% The case of $S_-$ is treated similarly.
\end{proof}

Now we again return to the notation from Sections~\ref{ss:ogr} and~\ref{ss:spinor}.

\begin{lemma}\label{lemma-same-p-type-b}
Let $P$ be a branching point. Then on $\OGr(k,V_P)$ one has
\begin{equation*}
\Ext^\bullet(S,S \otimes p_{2*}(\Lambda^m(\Sym^2\CU_k))) = 
\begin{cases}
\kk[-2m], & \text{for $m=0,1$,} \\
0, & \text{otherwise}.
\end{cases}
\end{equation*}
\end{lemma}
\begin{proof}
By~\eqref{eq:lm-s2}, \eqref{eq:p2-star}, and~\eqref{eq:gelfand-zeitlin} we have
\begin{equation*}
p_{2*}(\Lambda^m(\Sym^2\CU_k)) \cong 
\bigoplus_{\substack{\nu \in \You kk,\ \nu^T = \nu,\\|\nu| + \dlen(\nu) = 2m}} \quad
\bigoplus_{\nu'_1 \geq \gamma_1 \geq \nu'_2 \geq \ldots \geq \nu'_k \geq \gamma_k \geq \nu'_{k+1}} \Sigma^\gamma\CU_k'[-\dlen(\nu)],
\end{equation*}
where $\nu' = \expd01\nu = (\nu_1,\dots,\nu_s,s,\nu_{s+1},\dots,\nu_k)$ and $s = \dlen(\nu)$.
% 
% For convenience let us denote $\CU=\CU'_k$ the tautological subbundle on $\OGr(k,V_P)$.
% Recall that $p_2:\TOGr_P(k, V)\to \OGr(k,V_P)$ is the projection map if one sees $\TOGr_P(k,V)$ as the relative Grassmannian $\Gr_\OGr(k,V_P)(k, \CU\oplus \CO)$ and recall~\eqref{eq:lm-s2}.
%   
% A simple Borel--Bott--Weil computation shows that the only nontrivial higher direct image of $\Sigma^{\expd10\nu}\CU_k$ under $p_{2*}$ is 
% $R^\bullet p_{2*}\Sigma^{\expd10\nu}\CU_k = \Sigma^{{\expd01\nu}}(\CU\oplus\CO)[-\dlen(\nu)]$. 
% Thus, it is enough to be able to compute
% \begin{equation*}
% H^\bullet(\OGr(k,V_P),S^\vee\otimes S \otimes \Sigma^{{\expd01\nu}}(\CU\oplus\CO))
% \end{equation*}
% for all symmetric $\nu$. 
% Recall that there is a direct sum decomposition
% \begin{equation*}
% \Sigma^\beta(\CU\oplus\CO)=\bigoplus_{\gamma}\Sigma^\gamma \CU,\ 
% \text{for all $\beta_1 \geq \gamma_1 \geq \beta_2 \geq \ldots \geq \beta_k \geq \gamma_k \geq \beta_{k+1}$},
% \end{equation*}
% where $\beta = {\expd01\nu}$.
Note that $\gamma_i \ge \nu'_{i+1}$ implies an inequality of heights $h(\gamma) \ge h(\nu') - 1 = h(\nu)$, and $\nu'_1 \ge \gamma_1$ means $w(\gamma) \le w(\nu') = w(\nu)$.
Since $\nu$ is symmetric, we deduce $h(\gamma) \ge w(\gamma)$.
% $\beta = \expd01\nu$, we have $h(\beta) = w(\beta) + 1$, hence $h(\gamma) \ge w(\gamma)$.
Applying Lemma~\ref{lemma-wt} to $\Sigma^\gamma\CU'_k \otimes S^\vee \otimes S$, we conclude that the cohomology groups $H^\bullet(\OGr(k,V_P), \Sigma^\gamma\CU'_k \otimes S^\vee \otimes S)$ 
are nontrivial only when $\gamma = (0)$ and $\gamma = (1)$.
Furthermore, since $\dlen(\gamma) = \dlen(\nu') = \dlen(\nu)$, the case $\gamma = (0)$ corresponds to $\nu = (0)$, and gives
% \begin{equation*}
$\Hom(S, S) = \kk$.
% \end{equation*}
Analogously, the case $\gamma = (1)$ corresponds to $\nu' = (\nu'_1,\nu'_2)$ with $\nu'_2 \le 1$.
Since $\nu' = \expd01\nu$ and $\nu$ is symmetric, this means $\nu = (1)$, and gives
% \begin{equation*}
$\Ext^2(S, S \otimes p_{2*}(\Sym^2\CU_k)) = \kk$.
% \end{equation*}
Since these are the only possibilities, the lemma is proven.
% 
% and
% \begin{multline*}
% H^1(\TOGr(k,V_P), S^\vee \otimes S \otimes R^1p_{2*}\Sym^2\CU_k) = 
% H^1(\OGr(k,V_P), S^\vee \otimes S \otimes \Lambda^2(\CU \oplus \CO)) = \\
% H^1(\OGr(k,V_P), S^\vee \otimes S \otimes \CU) = \kk
% \end{multline*} 
% A typical spectral sequence argument finishes the proof.
\end{proof}

\bibliographystyle{abbrv}  
\bibliography{curves}

\end{document}